\documentclass[12pt,reqno]{amsart}%
\usepackage{palatino}
\usepackage{amssymb}
\usepackage{enumitem}
\usepackage{geometry}
\usepackage{graphicx}
\usepackage{microtype}
\usepackage{tikz}
\usepackage{xcolor}
\usepackage{amsmath}
\usepackage{amsfonts}%
\providecommand{\U}[1]{\protect\rule{.1in}{.1in}}

\newtheorem{theorem}{Theorem}[section]
\newtheorem{proposition}[theorem]{Proposition}
\newtheorem{corollary}[theorem]{Corollary}
\newtheorem{lemma}[theorem]{Lemma}
\newtheorem{definition}[theorem]{Definition}
\newtheorem{example}[theorem]{Example}
\newtheorem{remark}[theorem]{Remark}
\numberwithin{equation}{section}

\begin{document}
\title[Closed range and essential norms]{Closed Range and Essential Norms of Composition Operators on Weighted Dirichlet Spaces}

\author{Caixing Gu}
\address{Caixing Gu: Department of Mathematics, California Polytechnic State University, San Luis Obispo, California 93407, USA.}
\email{cgu@calpoly.edu}

\author{Li He*}
\address{Li He: School of Mathematics and Information Science, Guangzhou University, Guangzhou 510006, China.}
\email{helichangsha1986@163.com}

\author{Xiaofeng Wang}
\address{Xiaofeng Wang: School of Mathematics and Information Science, Guangzhou University, Guangzhou 510006, China.}
\email{wxf@gzhu.edu.cn}

\author{Yuanhao Yan*}
\address{Yuanhao Yan: School of Mathematics and Information Science, Guangzhou University, Guangzhou 510006, China.}
\email{18322912287@163.com}

	
	
\thanks{Li He is supported by NNSF of China (Grant No. 12371127, 12471119)}

\begin{abstract}
We study closed range and essential norms of bounded composition operators $C_\varphi$ on weighted Dirichlet spaces $\mathcal{D}_{\alpha}$. For $-1<\alpha<0$, we first consider maps whose images are obtained by removing a compact subset from a simply connected subdomain of $\mathbb{D}$. In this setting, closed range is equivalent to the Reverse Carleson property and the Reverse Carleson Disk Condition for the counting measure $\mu_{\varphi,\alpha}$, uniform weighted area density of the image, $\mathbb{T}\subset\overline{\varphi(\mathbb{D})}$, and the inclusion of an outer annulus in the image. We also obtain a closed range characterization under the Uniform Tail Condition. For all $\alpha>-1$, we establish two-sided essential norm estimates in terms of local averages of the generalized Nevanlinna counting function and an exact formula using boundary tail integrals. If the normalized counting density has vanishing oscillation at the boundary, we obtain exact formulas in terms of its boundary limsup, Berezin transform, and local averages.
\end{abstract}
\maketitle

\section{Introduction}

Let $\widehat{\mathbb{C}}=\mathbb{C}\cup\{\infty\}$ be the Riemann sphere, $\mathbb{D}$ the open unit disk in the complex plane $\mathbb{C}$,
$\mathbb{T}$ the unit circle, and $dA$ the normalized area measure on
$\mathbb{D}$. For $\alpha>-1$, define
\[
dA_{\alpha}(z) = (1+\alpha)(1-|z|^{2})^{\alpha}dA(z),
\]
and
\[
A_{\alpha}(E)=\int_{E}dA_{\alpha}(z)
\]
for measurable sets \(E\subset\mathbb{D}\). We also write
\[
A(E)=\int_E dA(z)=A_0(E)
\]
for the unweighted area of \(E\). The weighted Bergman space
\(A_{\alpha}^{2}\) consists of analytic functions \(g\) on \(\mathbb{D}\)
satisfying
\[
\|g\|_{A_{\alpha}^{2}}^{2} = \int_{\mathbb{D}}|g(z)|^{2}dA_{\alpha}%
(z)<\infty.
\]

The weighted Dirichlet space $\mathcal{D}_{\alpha}$ consists of analytic
functions $f$ on $\mathbb{D}$ satisfying
\[
\|f\|_{\mathcal{D}_{\alpha}}^{2}
=|f(0)|^{2}
+\int_{\mathbb{D}}|f^{\prime}(z)|^{2}\,dA_{\alpha}(z)
<\infty.
\]
We denote by
\[
\mathcal{D}_{\alpha,0} = \{f\in\mathcal{D}_{\alpha}:f(0)=0\}
\]
the closed codimension-one subspace of functions vanishing at the origin. Let
\[
P_{0}:\mathcal{D}_{\alpha}\longrightarrow\mathcal{D}_{\alpha,0}, \qquad
P_{0}f=f-f(0).
\]
Then $P_{0}$ is the orthogonal projection onto $\mathcal{D}_{\alpha,0}$, and
\[
\mathcal{D}_{\alpha}= \mathbb{C}\oplus\mathcal{D}_{\alpha,0}.
\]

Let $\varphi$ be a nonconstant analytic self-map of $\mathbb{D}$. The
composition operator induced by $\varphi$ is defined by
\[
C_{\varphi}f=f\circ\varphi.
\]

We first record a normalization for the closed range problem. Let $\Phi$ be a
nonconstant analytic self-map of $\mathbb{D}$, set $a=\Phi(0)$, and define
\[
\sigma_{a}(z)=\frac{a-z}{1-\overline{a}z}, \qquad\varphi=\sigma_{a}\circ\Phi.
\]
Then $\varphi(0)=0$. Since $\sigma_{a}$ is involutive, we have $\Phi
=\sigma_{a}\circ\varphi, $ and consequently
\[
C_{\Phi}=C_{\varphi}C_{\sigma_{a}}, \qquad C_{\varphi}=C_{\Phi}C_{\sigma_{a}%
}.
\]
Because $C_{\sigma_{a}}$ is bounded and invertible on $\mathcal{D}_{\alpha}$,
it follows that
\[
C_{\Phi}\text{ is bounded} \quad\Longleftrightarrow\quad C_{\varphi}\text{ is
bounded},
\]
and
\[
 C_{\Phi} \text{ has closed range on } \mathcal{D}_{\alpha}
\quad\Longleftrightarrow\quad
 C_{\varphi} \text{ has closed range on } \mathcal{D}_{\alpha}.
\]

Moreover, since $\sigma_{a}$ extends to a homeomorphism of $\overline
{\mathbb{D}}$ and maps $\mathbb{T}$ onto itself,
\[
\mathbb{T}\subset\overline{\Phi(\mathbb{D})} \quad\Longleftrightarrow
\quad\mathbb{T}\subset\overline{\varphi(\mathbb{D})}.
\]

The normalization also preserves the Compact-Hole Condition introduced
in Definition~\ref{def2.7}. Indeed, $\sigma_a$ maps simply connected
domains to simply connected domains and compact subsets to compact
subsets. For the statements below, the counting measures, density
conditions, and Uniform Tail Condition are formulated for the normalized
symbol $\varphi$.

Accordingly, throughout Section~2, unless otherwise stated, we work with a
normalized analytic self-map $\varphi$ satisfying $\varphi(0)=0. $ All
counting functions, measures, density conditions, and tail conditions
appearing in the closed range results of Section~2 are associated with this
normalized symbol.

If $C_{\varphi}$ is bounded on $\mathcal{D}_{\alpha}$ and $\varphi(0)=0$, then $\mathcal{D}_{\alpha,0}$ is invariant under $C_{\varphi}$. Let
\[
C_{\varphi,0} = C_{\varphi}\big|_{\mathcal{D}_{\alpha,0}}.
\]
Every function $f\in\mathcal{D}_{\alpha}$ admits the unique orthogonal
decomposition
\[
f=f(0)+(f-f(0)),
\]
where $f-f(0)\in\mathcal{D}_{\alpha,0}$. Thus, relative to the orthogonal
decomposition
\[
\mathcal{D}_{\alpha}= \mathbb{C}\oplus\mathcal{D}_{\alpha,0},
\]
we have
\[
C_{\varphi}= I_{\mathbb{C}}\oplus C_{\varphi,0}.
\]
Consequently,
\[
C_{\varphi}\text{ has closed range on }\mathcal{D}_{\alpha}\quad
\Longleftrightarrow\quad C_{\varphi,0} \text{ has closed range on }%
\mathcal{D}_{\alpha,0}.
\]

Since $\varphi$ is nonconstant, $C_{\varphi}$ is injective. Hence, whenever
$C_{\varphi}$ is bounded on $\mathcal{D}_{\alpha}$, the closedness of its
range is equivalent to $C_{\varphi}$ being bounded below.

The generalized Nevanlinna counting function is an important tool in the study
of composition operators. For $w\in\mathbb{D}$, define
\[
N_{\varphi,\alpha}(w) = \sum_{\varphi(z)=w}(1-|z|^{2})^{\alpha},
\]
where every preimage is counted according to multiplicity, and set
$N_{\varphi,\alpha}(w)=0$ when $w\notin\varphi(\mathbb{D})$. We also write
\[
n_{\varphi}(w) = \sum_{\varphi(z)=w}1
\]
for the number of preimages of $w$, counted with multiplicity; this number may
be infinite. Since $\varphi(0)=0$, the change of variables formula gives
\[
\|C_{\varphi}f\|_{\mathcal{D}_{\alpha}}^{2}
= |f(0)|^{2} + (1+\alpha)
\int_{\mathbb{D}}|f^{\prime}(w)|^{2}
N_{\varphi,\alpha}(w)\,dA(w),
\qquad f\in\mathcal{D}_{\alpha}.
\]
The analytic area formula applies to the nonnegative measurable weight $(1-|z|^{2})^{\alpha}$, with both sides of the above identity initially allowed to take the value $+\infty$, and the boundedness of $C_{\varphi}$ on $\mathcal{D}_{\alpha}$ ensures that both sides are finite for every $f\in\mathcal{D}_{\alpha}$. In particular, the first term vanishes for $f\in\mathcal{D}_{\alpha,0}$.
Let $\mu_{\varphi,\alpha}$ be the positive Borel measure on $\mathbb{D}$ given by
\[
d\mu_{\varphi,\alpha}(w)=N_{\varphi,\alpha}(w)\,dA(w).
\]

For $z,w\in\mathbb{D}$, let
\[
\rho(z,w) = \left|  \frac{z-w}{1-\overline zw}\right|  \quad\text{and}%
\quad\beta(z,w) = \operatorname{arctanh}\rho(z,w) = \frac12\log\frac
{1+\rho(z,w)}{1-\rho(z,w)}.
\]
For $r>0$, set
\[
D(z,r)=\{w\in\mathbb{D}:\beta(z,w)<r\}.
\]
We refer to $D(z,r)$ as the Bergman disk centered at $z$ with radius $r$. For $0<\eta<1$, set
\[
D_{\eta}(z)=\{w\in\mathbb{D}:\rho(z,w)<\eta\}.
\]
Thus $D(z,r)=D_{\tanh r}(z)$. For $0<t<1$, set $\mathbb{D}_{t}=\{z\in\mathbb{D}:|z|<t\}.$

A positive Borel measure $\mu$ on $\mathbb{D}$ is called a Reverse Carleson
measure for $A_{\alpha}^{2}$ if there exists $c>0$ such that
\[
\int_{\mathbb{D}}|g(w)|^{2}\,d\mu(w) \geq c\|g\|_{A_{\alpha}^{2}}^{2}, \qquad
g\in A_{\alpha}^{2}.
\]

We say that a positive Borel measure $\mu$ on $\mathbb{D}$ satisfies the
Reverse Carleson Disk Condition with respect to $A_{\alpha}$ if there
exist constants $\delta>0$ and $r>0$ such that
\[
\mu(D(z,r)) \geq\delta A_{\alpha}(D(z,r)), \qquad z\in\mathbb{D}.
\]
In \cite{MR4726060}, Reverse Carleson measure for the counting measure $n_{\varphi}\,dA$ refers to the existence of constants $\delta>0$ and $r>0$ such that
\[
\int_{D(z,r)}n_{\varphi}(w)\,dA(w)\geq\delta A(D(z,r)),\qquad z\in\mathbb{D}.
\]
This is precisely the Reverse Carleson Disk Condition in the unweighted case $\alpha=0$, since $N_{\varphi,0}=n_{\varphi}$ and $A_0=A$. In the present paper, the term ``Reverse Carleson measure for $A_{\alpha}^{2}$'' refers to the integral inequality above, required to hold for every $g\in A_{\alpha}^{2}$. These two notions are not equivalent in general. Thus the observation in \cite{MR4726060} that the Reverse Carleson Disk Condition alone need not imply closed range is consistent with Proposition~\ref{prop2.2}, which characterizes closed range using the integral inequality.

Composition operators on analytic function spaces have been extensively
studied, particularly with regard to their boundedness and compactness.
B. R. Choe, H. Koo, and W. Smith investigated these properties on holomorphic Sobolev
spaces and small spaces \cite{MR1975402,MR2270842}.
H. Wulan, D. Zheng, and K. Zhu obtained compactness criteria for composition
operators on BMOA and the Bloch space \cite{MR2529895}.
Background on analytic function spaces and operator theory can be found
in \cite{MR2537698,1}.

The closed range problem for composition operators is closely connected
with counting functions and Carleson-type inequalities.
D. H. Luecking's results on norm inequalities, dominating sets, and restriction
operators in weighted Bergman spaces
\cite{MR602889,MR722745,MR778090} provided a foundation for this approach.
N. Zorboska \cite{MR1236226} characterized the closed range property on
the Hardy space $H^2$ and weighted Bergman spaces through thickness
conditions for suitable level sets associated with counting functions.
Further developments in the weighted Bergman setting appear in
\cite{MR2872608}.
P. Ghatage, J. Yan, and D. Zheng \cite{MR1825915} studied geometric conditions
for the closed range property of composition operators on the Bloch space.

On the classical Dirichlet space, the connection between local measure
conditions and closed range is more delicate.
D. H. Luecking \cite{MR1637392} showed that the Reverse Carleson Disk Condition
for the counting measure does not by itself imply closed range.
N. Zorboska \cite{MR3998219} subsequently obtained a characterization
in terms of the Berezin transform under a vanishing oscillation assumption.
J. Pau and P. A. P\'erez \cite{MR3018017} obtained closed range criteria for
weighted Dirichlet spaces with $0<\alpha<1$ using generalized counting functions.
Further results on the closed range problem in Dirichlet and related
spaces can be found in \cite{MR4726060,MR3079836,MR3772158}.
For negative parameters, M. Jovovi\'{c} and B. MacCluer
\cite{MR1459789} related closed range to reverse Bergman inequalities and
proved that a univalent symbol inducing a bounded composition operator
with closed range must be an automorphism. Their boundary argument is
central to the geometric result below. We extend that argument to symbols
whose images have the form $U\setminus K$, where $U$ is simply connected
and $K$ is a compact subset of $U$, by transferring boundedness to a
conformal map onto $U$. This includes finitely connected images and
does not require $K$ to have finitely many connected components. We also study the Uniform Tail Condition as a separate route
to geometric characterizations for images of arbitrary connectivity.

For $r>0$, define
\[
\widetilde{\mu}_{\varphi,\alpha,r}(z) = \frac{\mu_{\varphi,\alpha}(D(z,r))}
{A_{\alpha}(D(z,r))}, \qquad z\in\mathbb{D}.
\]

We first record the standard equivalence between closed range of a bounded
normalized composition operator and the Reverse Carleson property of
$\mu_{\varphi,\alpha}$ for $A_{\alpha}^{2}$; see also
\cite{MR1459789} for negative parameters. Our geometric
characterizations use two different hypotheses: the Compact-Hole Condition
on the image, or the Uniform Tail Condition of the symbol.

The first main theorem, Theorem~\ref{thm2.14}, assumes $-1<\alpha<0$,
$\varphi(0)=0$, boundedness of $C_{\varphi}$ on $\mathcal{D}_{\alpha}$,
and the Compact-Hole Condition on $\Omega=\varphi(\mathbb{D})$: there exist
a simply connected domain $U\subseteq\mathbb{D}$ and a compact subset
$K\Subset U$ such that $\Omega=U\setminus K$. Under these hypotheses, the following
five conditions are equivalent:

\begin{enumerate}
[label=\textup{(\roman*)}]

\item $C_{\varphi}$ has closed range on $\mathcal{D}_{\alpha}$;

\item $\mu_{\varphi,\alpha}$ satisfies the Reverse Carleson Disk Condition;
that is, there exist $\delta>0$ and $r>0$ such that
\[
\mu_{\varphi,\alpha}(D(z,r)) \geq\delta A_{\alpha}(D(z,r)), \qquad
z\in\mathbb{D};
\]

\item there exist $\delta>0$ and $r>0$ such that
\[
A_{\alpha}(\varphi(\mathbb{D})\cap D(z,r)) \geq\delta A_{\alpha}(D(z,r)),
\qquad z\in\mathbb{D};
\]

\item there exists $r_{0}\in(0,1)$ such that
\[
\mathbb{D}\setminus\overline{\mathbb{D}_{r_{0}}} \subset\varphi(\mathbb{D});
\]

\item $\mu_{\varphi,\alpha}$ is a Reverse Carleson measure for $A_{\alpha}%
^{2}$.
\end{enumerate}

These conditions are also equivalent to
$\mathbb{T}\subset\overline{\Omega}$. The sufficiency of the annular inclusion
for closed range was established in \cite{MR1459789}.
To obtain the annular inclusion from boundary contact, we use the
decomposition $\Omega=U\setminus K$ in the Compact-Hole Condition.
Factoring $\varphi$ through a conformal map onto $U$ allows us to use
the boundary argument of \cite{MR1459789} and conclude that $U=\mathbb{D}$.
The compactness of $K$ then gives the required outer annulus.
The Compact-Hole Condition alone does not force closed range: the
dilation and power maps in Examples~\ref{ex2.16} and~\ref{ex2.17}
distinguish the two possibilities.

Our second main theorem extends the uniform tail characterization of
G. Cao and L. He \cite[Main Theorem~B]{MR4726060} to negative weighted Dirichlet
parameters, without imposing the Compact-Hole Condition on the image. Related
closed range results can be found in \cite{MR4060210}. Let $-1<\alpha<0$,
and let $\varphi$ be a nonconstant analytic self-map of $\mathbb{D}$
with $\varphi(0)=0$. For $M>0$, set
\[
E_{M} = \left\{  w\in\varphi(\mathbb{D}): \frac{N_{\varphi,\alpha}%
(w)}{(1-|w|^{2})^{\alpha}}>M \right\}  .
\]
We assume
\[
\lim_{M\to\infty} \sup_{\substack{f\in\mathcal{D}_{\alpha,0}%
\\\|f\|_{\mathcal{D}_{\alpha}}=1}} \int_{E_{M}}|f^{\prime}(w)|^{2}N_{\varphi,\alpha
}(w)\,dA(w) =0.
\]

Under this condition, $C_{\varphi}$ is bounded on $\mathcal{D}_{\alpha}$
by Lemma~\ref{lem2.21}. The closed range property of $C_{\varphi}$ is
equivalent to the Reverse Carleson Disk Condition, to the weighted area
density condition above, and to the Reverse Carleson property of
$\mu_{\varphi,\alpha}$.

The second part of the paper concerns essential norms. Since $\mathcal{D}%
_{\alpha,0}$ need not be invariant under $C_{\varphi}$ when $\varphi(0)\ne0$,
we consider an appropriate compression of $C_{\varphi}$ to $\mathcal{D}%
_{\alpha,0}$. Through the unitary differentiation map from $\mathcal{D}%
_{\alpha,0}$ onto $A_{\alpha}^{2}$, this compression is identified with a
weighted composition operator on $A_{\alpha}^{2}$. This allows us to reduce
the essential norm problem to that of a positive Toeplitz operator.

Essential norm estimates for weighted composition operators between
weighted Bergman spaces and between Hardy spaces were obtained in
\cite{MR2342670}, and for
Bergman spaces on strongly pseudoconvex domains in \cite{MR2329700}.
Related estimates for composition operators on weighted Dirichlet
spaces appear in \cite{MR3772661}.
For every $\alpha>-1$ and every bounded composition operator
$C_\varphi$ on $\mathcal{D}_\alpha$, Theorem~\ref{thm3.5}
applies S. Yamaji's estimate \cite[Theorem~B]{MR3127818}
to obtain a two-sided essential norm estimate in terms of local
averages of $N_{\varphi,\alpha}$.
Proposition~\ref{prop3.6} also gives the exact boundary tail formula
\[
\|C_\varphi\|_e^2
=(1+\alpha)\lim_{R\to1^-}
\sup_{\substack{g\in A_\alpha^2\\\|g\|_{A_\alpha^2}=1}}
\int_{|w|>R}|g(w)|^2N_{\varphi,\alpha}(w)\,dA(w).
\]
To obtain an exact local average formula, we impose a boundary
oscillation condition on the normalized counting density
$\psi(w)=N_{\varphi,\alpha}(w)/(1-|w|^2)^\alpha$.
If $\sup_{w\in D(z,s)}|\psi(w)-\psi(z)|\to0$ as $|z|\to1^-$ for every
fixed $s>0$, Theorem~\ref{thm3.8} yields
\[
\|C_\varphi\|_e^2
=\limsup_{|z|\to1^-}\psi(z)
=(1+\alpha)\limsup_{|z|\to1^-}
\frac{\displaystyle\int_{D(z,r)}N_{\varphi,\alpha}(w)\,dA(w)}
{A_\alpha(D(z,r))},\qquad r>0.
\]
The same quantity is the boundary limsup of the Berezin transform
of $\psi$. These formulas apply the boundary behavior of positive
Toeplitz operators to the counting density; related results for
symbols of vanishing oscillation are given in \cite{Hagger2017}.

The paper is organized as follows. Section~2 contains the closed range
characterizations under the Compact-Hole Condition (Theorem~\ref{thm2.14}) and
under the Uniform Tail Condition (Theorem~\ref{thm2.22}). Section~3
establishes the general essential norm estimate and the exact boundary
tail formula for all $\alpha>-1$, followed by exact local average
formulas under the boundary oscillation condition.

\section{Closed range of $C_{\varphi}$}
Throughout this section, we call $\varphi$ a normalized self-map of
$\mathbb{D}$ in $\mathcal{D}_{\alpha}$ if $\varphi$ is a nonconstant
analytic self-map of $\mathbb{D}$, $\varphi(0)=0$, and $C_{\varphi}$ is
bounded on $\mathcal{D}_{\alpha}$.

\subsection{Geometric necessary conditions for closed range of $C_{\varphi}$}\mbox{}

By the decomposition established in the Introduction, the closed range of $C_{\varphi}$ on $\mathcal{D}_{\alpha}$ is equivalent to that for its
restriction $C_{\varphi,0}$ to $\mathcal{D}_{\alpha,0}$. We may therefore
carry out the derivative based arguments below on $\mathcal{D}_{\alpha,0}$.

We first characterize closed range of $C_{\varphi}$ through the Reverse
Carleson property of $\mu_{\varphi,\alpha}$ for $A_{\alpha}^{2}$. We then
derive necessary geometric conditions, including the Reverse Carleson Disk
Condition and the requirement that the closure of $\varphi(\mathbb{D})$
contain the entire unit circle.

\begin{lemma}
\label{lem2.1} Let $\varphi$ be a nonconstant analytic self-map of
$\mathbb{D}$ with $\varphi(0)=0$. If $-1<\alpha\leq0$, then
\[
N_{\varphi,\alpha}(w)\geq n_{\varphi}(w)(1-|w|^{2})^{\alpha}, \qquad
w\in\mathbb{D }.
\]

\end{lemma}

\begin{proof}
By the Schwarz lemma, $|\varphi(z)|\leq|z|$ for all $z\in\mathbb{D}$. Thus,
whenever $\varphi(z)=w$, we have $|w|\leq|z|$, and hence
\[
1-|z|^{2}\leq1-|w|^{2} .
\]
Since $-1<\alpha\leq0$, the function $t\mapsto t^{\alpha}$ is nonincreasing
on $(0,\infty)$. Therefore,
\[
(1-|z|^{2})^{\alpha}\geq(1-|w|^{2})^{\alpha}%
\]
for every $z\in\mathbb{D}$ satisfying $\varphi(z)=w$. Consequently,
\[
N_{\varphi,\alpha}(w) =\sum_{\varphi(z)=w}(1-|z|^{2})^{\alpha}\geq
\sum_{\varphi(z)=w}(1-|w|^{2})^{\alpha}=n_{\varphi}(w)(1-|w|^{2})^{\alpha}.
\]
The proof is complete.
\end{proof}

\begin{proposition}
\label{prop2.2} Let $\alpha>-1$, and let $\varphi$ be a normalized
self-map of $\mathbb{D}$ in $\mathcal{D}_{\alpha}$. Then the following statements are equivalent:

\begin{enumerate}
[label=\textup{(\roman*)}]

\item $C_{\varphi}$ has closed range on $\mathcal{D}_{\alpha}$.

\item $\mu_{\varphi,\alpha}$ is a Reverse Carleson measure for $A_{\alpha}%
^{2}$.
\end{enumerate}
\end{proposition}

\begin{proof}
Since
\[
C_{\varphi}= I_{\mathbb{C}}\oplus C_{\varphi,0},
\]
it is enough to determine when $C_{\varphi,0}$ is bounded below. For
$f\in\mathcal{D}_{\alpha,0}$, put $g=f^{\prime}$. Then
\[
\|f\|_{\mathcal{D}_{\alpha}}^{2} = \int_{\mathbb{D}}|g(w)|^{2}\,dA_{\alpha
}(w),
\]
whereas
\[
\|C_{\varphi,0}f\|_{\mathcal{D}_{\alpha}}^{2} = (1+\alpha) \int_{\mathbb{D}%
}|g(w)|^{2}\,d\mu_{\varphi,\alpha}(w).
\]
Since differentiation is unitary from $\mathcal{D}_{\alpha,0}$ onto
$A_{\alpha}^{2}$, the conclusion follows.
\end{proof}

We next show that closed range implies the Reverse Carleson Disk Condition.
\begin{lemma}
\label{lem2.3} Let $\alpha>-1$ and let $\varphi$ be a normalized self-map of $\mathbb{D}$ in $\mathcal{D}_{\alpha}$. If $\mu_{\varphi,\alpha}$ is a Reverse Carleson measure for $A_{\alpha}^{2}$, then $\mu_{\varphi,\alpha}$ satisfies the Reverse Carleson Disk Condition; that is, there exist constants $\delta>0$ and $r>0$ such that
\[
\mu_{\varphi,\alpha}(D(z,r))\geq\delta A_{\alpha}(D(z,r)),\qquad z\in\mathbb{D}.
\]
\end{lemma}

\begin{proof}
By assumption, $\mu_{\varphi,\alpha}$ is a Reverse Carleson measure for $A_{\alpha}^{2}$. Hence there exists a constant $c_{0}>0$ such that
\begin{equation}
\int_{\mathbb{D}}|g(w)|^{2}\,d\mu_{\varphi,\alpha}(w) \geq c_{0}
\int_{\mathbb{D}}|g(w)|^{2}\,dA_{\alpha}(w), \qquad g\in A_{\alpha}^{2}.
\label{eq:2.1}
\end{equation}
Taking $g=\kappa_{z}^{\alpha}$, where
\[
\kappa_{z}^{\alpha}(w) =\frac{(1-|z|^{2})^{\frac{\alpha+2}{2}}} {(1-\overline
zw)^{\alpha+2}}
\]
is the normalized reproducing kernel of $A_{\alpha}^{2}$, we obtain
\begin{equation}
\int_{\mathbb{D}} |\kappa_{z}^{\alpha}(w)|^{2}\,d\mu_{\varphi,\alpha}(w) \geq
c_{0}, \qquad z\in\mathbb{D}. \label{eq:2.2}
\end{equation}

On the other hand, for $g\in A_{\alpha}^{2}$, define
\[
F(z)=\int_{0}^{z} g(\xi)\,d\xi.
\]
Then
\[
F^{\prime}(z)=g(z)\qquad\text{and}\qquad F \in\mathcal{D}_{\alpha,0}.
\]
By the change of variables formula,
\[
\|C_{\varphi}F\|_{\mathcal{D}_{\alpha}}^{2} = (1+\alpha) \int_{\mathbb{D}}
|g(w)|^{2}N_{\varphi,\alpha}(w)\,dA(w),
\]
while,
\[
\|F\|_{\mathcal{D}_{\alpha}}^{2} = \int_{\mathbb{D}}|g(w)|^{2}\,dA_{\alpha
}(w).
\]
Then the boundedness of $C_{\varphi}$ implies that
\begin{equation}
\int_{\mathbb{D}}|g(w)|^{2}\,d\mu_{\varphi,\alpha}(w) \lesssim\int
_{\mathbb{D}}|g(w)|^{2}\,dA_{\alpha}(w). \label{eq:2.3}
\end{equation}
This shows that $\mu_{\varphi,\alpha}$ is an $A_{\alpha}^{2}$-Carleson measure.

Fix $0 < s < 1$. We claim that
\begin{equation}
\mu_{\varphi,\alpha}(D(a,s)) \lesssim A_{\alpha}(D(a,s)), \qquad
a\in\mathbb{D}. \label{eq:2.4}
\end{equation}
Indeed, taking $g=\kappa_{a}^{\alpha}$ in \eqref{eq:2.3}, we obtain
\[
\int_{\mathbb{D}}|\kappa_{a}^{\alpha}(w)|^{2}\,d\mu_{\varphi,\alpha}(w)
\lesssim1.
\]
This implies that
\[
\int_{D(a,s)} |\kappa_{a}^{\alpha}(w)|^{2}\,d\mu_{\varphi,\alpha}(w)
\lesssim1.
\]
For $w\in D(a,s)$, we have
\[
|1-\overline aw| \asymp1-|a|^{2},
\]
and hence
\[
|\kappa_{a}^{\alpha}(w)|^{2} \asymp\frac{1}{(1-|a|^{2})^{\alpha+2}}
\asymp\frac{1}{A_{\alpha}(D(a,s))}.
\]
Consequently,
\[
1 \gtrsim\int_{D(a,s)} |\kappa_{a}^{\alpha}(w)|^{2}\,d\mu_{\varphi,\alpha}(w)
\gtrsim\frac{\mu_{\varphi,\alpha}(D(a,s))} {A_{\alpha}(D(a,s))},
\]
which proves \eqref{eq:2.4}.

Let $\{a_{k}\}$ be an $s$-lattice such that
\[
\mathbb{D}=\bigcup_{k}D(a_{k},s),
\]
By \cite[Theorem~2.23]{MR2115155}, there exists an integer $N$, independent of $w$, such that
\begin{equation}
\sum_{k}\chi_{D(a_{k},4s)}(w) \leq N, \qquad w\in\mathbb{D}. \label{eq:2.5}
\end{equation}
For $r>5s$, define
\[
E_{r}(z) = \mathbb{D}\setminus D(z,r)
\]
and
\[
J_{r}(z) = \left\{  k: D(a_{k},s)\cap E_{r}(z)\neq\emptyset\right\}  .
\]
Then
\begin{align*}
&  \int_{E_{r}(z)} |\kappa_{z}^{\alpha}(w)|^{2}\,d\mu_{\varphi,\alpha}(w)\\
&  \quad\leq\sum_{k\in J_{r}(z)} \int_{D(a_{k},s)} |\kappa_{z}^{\alpha
}(w)|^{2}\,d\mu_{\varphi,\alpha}(w)\\
&  \quad\leq\sum_{k\in J_{r}(z)} \mu_{\varphi,\alpha}(D(a_{k},s)) \sup_{w\in
D(a_{k},s)} |\kappa_{z}^{\alpha}(w)|^{2}.
\end{align*}

By the submean property for analytic functions, we have
\[
|f(w)|^{2} \lesssim\frac{1}{A_{\alpha}(D(w,s))} \int_{D(w,s)} |f(u)|^{2}%
\,dA_{\alpha}(u)\quad\text{for}\quad w \in D\left(  {{a_{k}},s} \right)  .
\]
Since
\[
D(w,s)\subset D(a_{k},4s)
\]
and
\[
A_{\alpha}(D(w,s)) \asymp A_{\alpha}(D(a_{k},s)), \qquad w\in D(a_{k},s),
\]
it follows that
\[
|f(w)|^{2} \lesssim\frac{1}{A_{\alpha}(D(a_{k},s))} \int_{D(a_{k},4s)}
|f(u)|^{2}\,dA_{\alpha}(u).
\]
Applying this to $f=\kappa_{z}^{\alpha}$ and taking the supremum over $w\in
D(a_{k},s)$, we obtain
\begin{equation}
\sup_{w\in D(a_{k},s)} |\kappa_{z}^{\alpha}(w)|^{2} \lesssim\frac{1}%
{A_{\alpha}(D(a_{k},s))} \int_{D(a_{k},4s)} |\kappa_{z}^{\alpha}%
(u)|^{2}\,dA_{\alpha}(u). \label{eq:2.6}
\end{equation}
Combining \eqref{eq:2.4} and \eqref{eq:2.6}, we get
\begin{equation}
\int_{E_{r}(z)} |\kappa_{z}^{\alpha}(w)|^{2}\,d\mu_{\varphi,\alpha}(w)
\lesssim\sum_{k\in J_{r}(z)} \int_{D(a_{k},4s)} |\kappa_{z}^{\alpha}%
(u)|^{2}\,dA_{\alpha}(u). \label{eq:2.7}
\end{equation}
For each $k\in J_{r}(z)$, choose
\[
\xi_{k} \in D(a_{k},s)\cap\bigl(\mathbb{D}\setminus D(z,r)\bigr).
\]
Then $\beta(a_{k},\xi_{k})<s $ and $\beta(z,\xi_{k})\geq r. $ If $u\in
D(a_{k},4s)$, then
\[
\beta(u,a_{k})<4s,
\]
and hence
\[
\beta(u,\xi_{k}) \leq\beta(u,a_{k})+\beta(a_{k},\xi_{k}) < 4s+s = 5s.
\]
By the reverse triangle inequality,
\[
\beta(z,u) \geq\beta(z,\xi_{k})-\beta(\xi_{k},u) > r-5s.
\]
Therefore,
\begin{equation}
D(a_{k},4s) \subset E_{r-5s}(z), \qquad k\in J_{r}(z). \label{eq:2.8}
\end{equation}
Combining \eqref{eq:2.5}, \eqref{eq:2.7}, and \eqref{eq:2.8}, we obtain
\begin{align}
\int_{E_{r}(z)} |\kappa_{z}^{\alpha}(w)|^{2}\,d\mu_{\varphi,\alpha}(w)  &
\lesssim\sum_{k\in J_{r}(z)} \int_{D(a_{k},4s)} |\kappa_{z}^{\alpha}%
(u)|^{2}\,dA_{\alpha}(u)\notag\\
&  =\int_{\mathbb{D}} |\kappa_{z}^{\alpha}(u)|^{2} \sum_{k\in J_{r}(z)}
\chi_{D(a_{k},4s)}(u)\,dA_{\alpha}(u)\notag\\
&  =\int_{E_{r-5s}(z)} |\kappa_{z}^{\alpha}(u)|^{2} \sum_{k\in J_{r}(z)}
\chi_{D(a_{k},4s)}(u)\,dA_{\alpha}(u)\notag\\
&  \leq N\int_{E_{r-5s}(z)} |\kappa_{z}^{\alpha}(u)|^{2}\,dA_{\alpha}(u).
\label{eq:2.9}%
\end{align}

We now compute the integral on the right-hand side. Set
\[
R_{1}=\tanh(r-5s),
\]
then
\[
D(z,r-5s)=D_{R_{1}}(z).
\]

Let
\[
\sigma_{z}(u) = \frac{z-u}{1-\overline zu}.
\]
Then
\[
|\kappa_{z}^{\alpha}(\sigma_{z}(u))|^{2} \,dA_{\alpha}(\sigma_{z}(u)) =
dA_{\alpha}(u).
\]
Therefore,
\[
\int_{D_{R_{1}}(z)} |\kappa_{z}^{\alpha}(w)|^{2}\,dA_{\alpha}(w) =
\int_{|u|<R_{1}}dA_{\alpha}(u) = 1-(1-R_{1}^{2})^{\alpha+1}.
\]
It follows that
\begin{equation}
\int_{\mathbb{D}\setminus D_{R_{1}}(z)} |\kappa_{z}^{\alpha}(w)|^{2}%
\,dA_{\alpha}(w) = (1-R_{1}^{2})^{\alpha+1}. \label{eq:2.10}
\end{equation}
Combining \eqref{eq:2.9} and \eqref{eq:2.10}, we get
\begin{equation}
\int_{E_{r}(z)} |\kappa_{z}^{\alpha}(w)|^{2} d\mu_{\varphi,\alpha}(w)
\lesssim(1-R_{1}^{2})^{\alpha+1}. \label{eq:2.11}
\end{equation}
Since $\alpha>-1$, the right-hand side tends to zero as $r\to\infty$.
Therefore,
\[
\lim_{r\to+\infty} \sup_{z\in\mathbb{D}} \int_{E_{r}(z)} |\kappa_{z}^{\alpha
}(w)|^{2} d\mu_{\varphi,\alpha}(w) = 0.
\]
Therefore, we may choose $r>5s$ sufficiently large such that
\begin{equation}
\int_{E_{r}(z)} |\kappa_{z}^{\alpha}(w)|^{2}\,d\mu_{\varphi,\alpha}(w)
\leq\frac{c_{0}}{2}, \qquad z\in\mathbb{D}. \label{eq:2.12}
\end{equation}

Combining \eqref{eq:2.2} and \eqref{eq:2.12}, we obtain
\begin{align*}
\int_{D(z,r)} |\kappa_{z}^{\alpha}(w)|^{2} d\mu_{\varphi,\alpha}(w)  &  =
\int_{\mathbb{D}} |\kappa_{z}^{\alpha}(w)|^{2} d\mu_{\varphi,\alpha}(w)-
\int_{E_{r}(z)} |\kappa_{z}^{\alpha}(w)|^{2} d\mu_{\varphi,\alpha}(w)\\
&  \geq\frac{c_{0}}{2}.
\end{align*}
Therefore,
\begin{align*}
\frac{c_{0}}{2}  &  \leq\int_{D(z,r)} |\kappa_{z}^{\alpha}(w)|^{2}%
\,d\mu_{\varphi,\alpha}(w)\lesssim\frac{\mu_{\varphi,\alpha}(D(z,r))}
{A_{\alpha}(D(z,r))}%
\end{align*}
by the equivalence
\[
|\kappa_{z}^{\alpha}(w)|^{2} \asymp\frac{1}{(1-|z|^{2})^{\alpha+2}}
\asymp\frac{1}{A_{\alpha}(D(z,r))},\qquad w\in D(z,r).
\]
Consequently, there exists a constant $\delta>0$, independent of $z$, such
that
\[
\mu_{\varphi,\alpha}(D(z,r)) \geq\delta A_{\alpha}(D(z,r)), \qquad
z\in\mathbb{D}.
\]
The proof is complete.
\end{proof}

\begin{theorem}
\label{thm2.4} Let $\alpha>-1$ and let $\varphi$ be a normalized self-map of $\mathbb{D}$ in $\mathcal{D}_{\alpha}$. If the range of $C_{\varphi}$ is closed in $\mathcal{D}_{\alpha}$,
then $\mu_{\varphi,\alpha}$ satisfies the Reverse Carleson Disk Condition.
\end{theorem}

\begin{proof}
By Proposition~\ref{prop2.2}, $\mu_{\varphi,\alpha}$ is a Reverse Carleson measure for $A_{\alpha}^{2}$. The conclusion follows from Lemma~\ref{lem2.3}.
\end{proof}

Theorem~\ref{thm2.4} shows that closed range implies the Reverse Carleson Disk
Condition for $\mu_{\varphi,\alpha}$. In the following lemmas, we use this
condition as a starting point to derive geometric and measure theoretic consequences.

\begin{lemma}
\label{lem2.5} Let $\alpha>-1$, $r>0$, and let $G$ be a Borel subset of
$\mathbb{D}$. Then the following two conditions are equivalent:
\begin{equation}
A_{\alpha}(G\cap D(z,r))\geq\delta A_{\alpha}(D(z,r)), \qquad z\in
\mathbb{D},\label{eq:2.13}
\end{equation}
and
\begin{equation}
A(G\cap D(z,r))\geq\delta_{1} A(D(z,r)), \qquad z\in\mathbb{D},\label{eq:2.14}
\end{equation}
where $\delta,\delta_{1}>0$ are constants independent of $z$.
\end{lemma}

\begin{proof}
Since $r>0$ is fixed, for every $w\in D(z,r)$ we have
\[
1-|w|^{2}\asymp1-|z|^{2},
\]
where the constants depend only on $r$. It follows that
\[
(1-|w|^{2})^{\alpha}\asymp(1-|z|^{2})^{\alpha}, \qquad w\in D(z,r).
\]
Therefore, for every measurable set $E\subset D(z,r)$,
\[
A_{\alpha}(E) = (1+\alpha)\int_{E} (1-|w|^{2})^{\alpha}\,dA(w) \asymp
(1-|z|^{2})^{\alpha}A(E).
\]
Applying this to $E = G \cap D\left(  {z,r} \right)  $ and $E = D\left(  {z,r}
\right)  $, we obtain
\[
A_{\alpha}(G\cap D(z,r)) \asymp(1-|z|^{2})^{\alpha}A(G\cap D(z,r)),
\]
and
\[
A_{\alpha}(D(z,r)) \asymp(1-|z|^{2})^{\alpha}A(D(z,r)).
\]
Consequently,
\[
\frac{{{A_{\alpha}}\left(  {G \cap D\left(  {z,r} \right)  } \right)  }%
}{{{A_{\alpha}}\left(  {D\left(  {z,r} \right)  } \right)  }} \asymp
\frac{{A\left(  {G \cap D\left(  {z,r} \right)  } \right)  }}{{A\left(
{D\left(  {z,r} \right)  } \right)  }},\qquad z \in\mathbb{D}%
\]
Thus, the condition \eqref{eq:2.13} holds if and only if the
condition \eqref{eq:2.14} holds, with comparable constants.

The proof is complete.
\end{proof}

\begin{theorem}
\label{thm2.6} Let $-1<\alpha\leq0$, and let $\varphi$ be a normalized
self-map of $\mathbb{D}$ in $\mathcal{D}_{\alpha}$. If
$C_{\varphi}$ has closed range, then
\[
\mathbb{T}\subset\overline{\varphi(\mathbb{D})}.
\]
Equivalently, for every $\zeta\in\mathbb{T}$ and every $r>0$,
\[
S(\zeta,r)\cap\varphi(\mathbb{D})\ne\emptyset, \qquad S(\zeta,r)=\{z\in
\mathbb{D}:|z-\zeta|<r\}.
\]

\end{theorem}

\begin{proof}
Since $C_{\varphi}$ is injective and has closed range, there exists $c>0$ such
that
\[
\|C_{\varphi}f\|_{\mathcal{D}_{\alpha}} \geq c\|f\|_{\mathcal{D}_{\alpha}},
\qquad f\in\mathcal{D}_{\alpha}.
\]
Suppose, to the contrary, that there are $\zeta\in\mathbb{T}$ and an open
neighborhood $U$ of $\zeta$ such that $U\cap\varphi(\mathbb{D})=\emptyset$.
Put
\[
p_{\zeta}(z)=\frac{1+\overline\zeta z}{2}, \qquad F_{k}(z)=p_{\zeta}%
(z)^{k}-p_{\zeta}(0)^{k}.
\]
Then $F_{k}\in\mathcal{D}_{\alpha,0}$ and
\[
F_{k}^{\prime}(z)=k p_{\zeta}(z)^{k-1}p_{\zeta}^{\prime}(z), \qquad|p_{\zeta
}^{\prime}(z)|=\frac12.
\]
The function $p_{\zeta}$ peaks at $\zeta$; hence
\[
q:=\sup_{w\in\varphi(\mathbb{D})}|p_{\zeta}(w)| \leq\max_{w\in\overline
{\mathbb{D}}\setminus U}|p_{\zeta}(w)| <1.
\]
Moreover,
\[
\lim_{k\to\infty} \left(  \int_{\mathbb{D}}|F_{k}^{\prime}(z)|^{2}\,dA_{\alpha}(z)
\right)  ^{1/(2k)} =1.
\]
Indeed, the upper bound follows from $|p_{\zeta}|\leq1$, while the lower bound
follows by integrating over the nonempty open set $\{z:|p_{\zeta
}(z)|>1-\varepsilon\}$ and then letting $\varepsilon\downarrow0$.

Since $C_{\varphi}$ is bounded, applying it to the identity function shows
that
\[
\int_{\varphi(\mathbb{D})}N_{\varphi,\alpha}(w)\,dA(w)<\infty.
\]
Using the change of variables formula and $\varphi(0)=0$, we obtain
\begin{align*}
\|C_{\varphi}F_{k}\|_{\mathcal{D}_{\alpha}}^{2}  &  =(1+\alpha) \int
_{\varphi(\mathbb{D})} |F_{k}^{\prime}(w)|^{2}N_{\varphi,\alpha}(w)\,dA(w)\\
&  \leq\frac{1+\alpha}{4}k^{2}q^{2k-2} \int_{\varphi(\mathbb{D})}%
N_{\varphi,\alpha}(w)\,dA(w).
\end{align*}
It follows that
\[
\frac{\|C_{\varphi}F_{k}\|_{\mathcal{D}_{\alpha}}} {\|F_{k}\|_{\mathcal{D}%
_{\alpha}}} \longrightarrow0,
\]
contradicting the lower bound for $C_{\varphi}$. Therefore every neighborhood
of every point of $\mathbb{T}$ meets $\varphi(\mathbb{D})$, and the conclusion follows.
\end{proof}

Theorem~\ref{thm2.6} gives a necessary boundary condition for closed
range. We next show that this condition is sufficient when the image
is obtained by removing a compact subset from a simply connected domain.
\subsection{The Compact-Hole Condition and closed range of $C_{\varphi}$}\mbox{}

We assume that $\varphi(\mathbb{D})=U\setminus K$, where $U$ is simply
connected and $K$ is a compact subset of $U$. We show that boundedness
of $C_{\varphi}$ passes to the composition operator induced by a
conformal map onto $U$. The boundary argument of \cite{MR1459789} then
turns contact with the entire unit circle into the inclusion of an
actual outer annulus in $\varphi(\mathbb{D})$.

This gives the five equivalent conditions in Theorem~\ref{thm2.14}.
The argument uses the compactness of $K$, without any restriction on
its number of connected components. The examples below show that the
geometric assumption does not force closed range and that boundedness
is essential for the annular conclusion.

\begin{definition}[Compact-Hole Condition]
\label{def2.7}
A domain $\Omega\subseteq\mathbb{D}$ satisfies the Compact-Hole
Condition if there exist a simply connected domain
$U\subseteq\mathbb{D}$ and a compact subset $K$ of $U$ such that
\[
\Omega=U\setminus K.
\]
We write $K\Subset U$ to indicate that $K$ is a compact subset of $U$.
The set $K$ may be empty and need not have finitely many connected
components.
\end{definition}

Every finitely connected domain satisfies the Compact-Hole Condition.
Indeed, if $\widehat{\mathbb{C}}\setminus\Omega$ has finitely many
connected components $K_0,K_1,\ldots,K_m$, let $K_0$ be the outer
complementary component containing $\infty$, and set
\[
U=\widehat{\mathbb{C}}\setminus K_0,
\qquad K=\bigcup_{j=1}^{m}K_j.
\]
Then $U$ is the simply connected domain obtained by filling the bounded
complementary components of $\Omega$. Each $K_j$, $1\leq j\leq m$,
is compact and contained in $U$, so $K\Subset U$ and $\Omega=U\setminus K$.
In particular, a simply connected domain satisfies the condition with
$U=\Omega$ and $K=\emptyset$.

We first record an estimate that will allow us to recover the full
weighted Bergman norm from integration over an outer annulus.
\begin{lemma}
\label{lem2.8} Let $\alpha>-1$ and $0<r<1$. For every analytic function
$g$ on $\mathbb{D}$,
\[
\int_{\mathbb{D}\setminus\overline{\mathbb{D}_{r}}}|g(w)|^{2}
\,dA_{\alpha}(w)
\geq (1-r^{2})^{\alpha+1}
\int_{\mathbb{D}}|g(w)|^{2}\,dA_{\alpha}(w).
\]
The integrals are understood in $[0,+\infty]$.
\end{lemma}

\begin{proof}
Set
\[
m_g(t)=\frac{1}{2\pi}\int_{0}^{2\pi}|g(te^{i\theta})|^{2}\,d\theta,
\qquad q=(1-r^{2})^{\alpha+1}.
\]
The function $m_g$ is nondecreasing. Since the weighted areas of
$\mathbb{D}_{r}$ and $\mathbb{D}\setminus\overline{\mathbb{D}_{r}}$
are $1-q$ and $q$, respectively, radial integration gives
\[
\int_{\mathbb{D}_{r}}|g|^{2}\,dA_{\alpha}
\leq (1-q)m_g(r)
\leq \frac{1-q}{q}
\int_{\mathbb{D}\setminus\overline{\mathbb{D}_{r}}}|g|^{2}\,dA_{\alpha}.
\]
If the last integral is finite, rearranging proves the assertion.
If it is infinite, the asserted inequality holds in the extended sense.
\end{proof}

The following consequence of the boundary argument in
\cite{MR1459789} is the analytic ingredient needed for the filling
construction. We include the argument to make clear that, once the
boundary inclusion is known, closed range is not an additional hypothesis.

\begin{lemma}
\label{lem2.9} Let $-1<\alpha<0$, and let $\psi$ be a univalent
analytic self-map of $\mathbb{D}$. Suppose that
$C_{\psi}$ is bounded on $\mathcal{D}_{\alpha}$ and
\[
\mathbb{T}\subset\overline{\psi(\mathbb{D})}.
\]
Then $\psi$ is an automorphism of $\mathbb{D}$.
\end{lemma}

\begin{proof}
The boundedness of $C_\psi$ gives $\psi\in\mathcal{D}_\alpha$. Since every function in $\mathcal{D}_\alpha$ extends continuously to $\overline{\mathbb{D}}$ when $-1<\alpha<0$, the set
\[
A=\{\zeta\in\mathbb{T}:|\psi(\zeta)|=1\}
\]
is compact. By \cite[Theorem~4.8]{MR1397026}, $\psi$ has a finite angular derivative at every point of $A$.

We first show that $\psi(A)=\mathbb{T}$. Let $\eta\in\mathbb{T}$. Since $\mathbb{T}\subset\overline{\psi(\mathbb{D})}$, there exist $z_n\in\mathbb{D}$ such that $\psi(z_n)\rightarrow\eta$.
Passing to a subsequence, we may assume that $z_n\to\zeta\in\overline{\mathbb{D}}$. Continuity gives $\psi(\zeta)=\eta$. Since $\psi(\mathbb{D})\subset\mathbb{D}$ and $|\eta|=1$, we must have $\zeta\in\mathbb{T}$. Thus $\zeta\in A$, and consequently $\mathbb{T}\subseteq\psi(A)$. The reverse inclusion follows from the definition of $A$, so we have $\psi(A)=\mathbb{T}.$

Next, suppose that $\zeta_1,\zeta_2\in A$ are distinct and satisfy $\psi(\zeta_1)=\psi(\zeta_2)\in\mathbb{T}.$
By \cite[Lemma~3.3]{MR1459789}, univalence implies that $\psi$ cannot have finite angular derivatives at each of these two points. This contradicts the property established above. Hence $\psi|_A$ is injective. Since $A$ is compact, the continuous bijection
\[
\psi|_A:A\longrightarrow\mathbb{T}
\]
is a homeomorphism.

We claim that $A=\mathbb{T}$. Otherwise, choose $\xi\in\mathbb{T}\setminus A$. The punctured circle $\mathbb{T}\setminus\{\xi\}$ is homeomorphic to $\mathbb{R}$, so $A$ is homeomorphic to a compact subset of $\mathbb{R}$. Since $A$ is connected and contains more than one point, this subset is a nondegenerate closed interval. Such an interval cannot be homeomorphic to $\mathbb{T}$.

Finally, we show that $\psi(\mathbb{D})=\mathbb{D}$. If $\psi(\mathbb{D})$ were a proper subdomain of $\mathbb{D}$, there would exist
\[
w\in\partial\psi(\mathbb{D})\cap\mathbb{D}.
\]
By compactness and continuity, $w=\psi(\zeta)$ for some $\zeta\in\overline{\mathbb{D}}$. Since $\psi(\mathbb{D})$ is open, its boundary point $w$ does not belong to $\psi(\mathbb{D})$, so $\zeta\notin\mathbb{D}$. Thus $\zeta\in\mathbb{T}=A$, which implies
\[
|w|=|\psi(\zeta)|=1.
\]
This contradicts $w\in\mathbb{D}$. Hence $\psi(\mathbb{D})=\mathbb{D}$, and univalence shows that $\psi$ is an automorphism of $\mathbb{D}$.
\end{proof}

For $\alpha>-1$ and $1\leq p<\infty$, the weighted Bergman space $A_{\alpha}^{p}$ consists of all analytic functions $f$ on $\mathbb{D}$ such that
\[
\|f\|_{A_{\alpha}^{p}}^{p}=\int_{\mathbb{D}}|f(z)|^{p}\,dA_{\alpha}(z)<\infty.
\]

\begin{lemma}
\label{lem2.10} Let $\alpha>-1$, $1\leq p<\infty$, and let $G$ be a Borel
subset of $\mathbb{D}$. The following statements are equivalent:

\noindent\textup{(i)} There exists a constant $C>0$ such that
\[
\int_{G} |f(z)|^{p}\,dA_{\alpha}(z) \geq C\int_{\mathbb{D}}|f(z)|^{p}%
\,dA_{\alpha}(z), \qquad f\in A_{\alpha}^{p}.
\]

\noindent\textup{(ii)} There exist constants $\delta>0$ and $r>0$ such that
\[
A_{\alpha}\bigl(G\cap D(z,r)\bigr)
\geq\delta A_{\alpha}\bigl(D(z,r)\bigr), \qquad z\in\mathbb{D}.
\]

\end{lemma}

\begin{proof}
D. H. Luecking's Theorem~1 and the converse proved in Section~2 of
\cite{MR722745}, applied to the unit disk with weight
$(1-|z|^{2})^{\alpha}$, show that condition \textup{(i)} is
equivalent to the existence of constants $\delta_{0}>0$ and
$r>0$ such that
\[
A\bigl(G\cap D(z,r)\bigr)
\geq\delta_{0}A\bigl(D(z,r)\bigr), \qquad z\in\mathbb{D}.
\]
By Lemma~\ref{lem2.5}, this unweighted density condition is equivalent to the
weighted density condition \textrm{(ii)}, with comparable constants. Hence the
conclusion follows.

The proof is complete.
\end{proof}

The next lemma combines the geometric density condition with the basic
equivalence between closed range and the Reverse Carleson property established above.

\begin{lemma}
\label{lem2.11} Let $-1<\alpha\leq0$, and let $\varphi$ be a normalized
self-map of $\mathbb{D}$ in $\mathcal{D}_{\alpha}$. Suppose that
there exist constants $\delta>0$ and $r>0$ such that
\[
A_{\alpha}\bigl(\varphi(\mathbb{D})\cap D(z,r)\bigr)
\geq\delta A_{\alpha}\bigl(D(z,r)\bigr), \qquad z\in\mathbb{D}.
\]
Then $\mu_{\varphi,\alpha}$ is a Reverse Carleson measure for $A_{\alpha}^{2}%
$. Consequently, if $C_{\varphi}$ is bounded on $\mathcal{D}_{\alpha}$, then
$C_{\varphi}$ has closed range on $\mathcal{D}_{\alpha}$.
\end{lemma}

\begin{proof}
By Lemma~\ref{lem2.1},
\[
N_{\varphi,\alpha}(w) \geq(1-|w|^{2})^{\alpha}, \qquad w\in\varphi
(\mathbb{D}).
\]
Hence, for every $g\in A_{\alpha}^{2}$,
\begin{align*}
\int_{\mathbb{D}}|g(w)|^{2}\,d\mu_{\varphi,\alpha}(w)  &  \geq\int
_{\varphi(\mathbb{D})} |g(w)|^{2}(1-|w|^{2})^{\alpha}\,dA(w)\\
&  = \frac{1}{1+\alpha} \int_{\varphi(\mathbb{D})}|g(w)|^{2}\,dA_{\alpha}(w).
\end{align*}
By Lemma~\ref{lem2.10}, the density hypothesis implies that there exists $C>0$
such that
\[
\int_{\varphi(\mathbb{D})}|g(w)|^{2}\,dA_{\alpha}(w) \geq C\|g\|_{A_{\alpha
}^{2}}^{2}, \qquad g\in A_{\alpha}^{2}.
\]
Therefore,
\[
\int_{\mathbb{D}}|g(w)|^{2}\,d\mu_{\varphi,\alpha}(w) \geq\frac{C}{1+\alpha
}\|g\|_{A_{\alpha}^{2}}^{2},
\]
so $\mu_{\varphi,\alpha}$ is a Reverse Carleson measure for $A_{\alpha}^{2}$.
If $C_{\varphi}$ is bounded on $\mathcal{D}_{\alpha}$, the closed range
conclusion follows from Proposition~\ref{prop2.2}.
\end{proof}

\begin{lemma}
\label{lem2.12} Let $-1<\alpha<0$, and let $\varphi$ be a normalized
self-map of $\mathbb{D}$ in $\mathcal{D}_{\alpha}$ with
\[
\varphi(\mathbb{D})=U\setminus K,
\]
where $U\subseteq\mathbb{D}$ is simply connected and $K\Subset U$.
If $\psi$ is a conformal bijection of $\mathbb{D}$ onto $U$ with $\psi(0)=0$, then
$C_{\psi}$ is bounded on $\mathcal{D}_{\alpha}$.
\end{lemma}

\begin{proof}
Set $h=\psi^{-1}\circ\varphi$. Then $h(0)=0$ and
\[
h(\mathbb{D})=\mathbb{D}\setminus\psi^{-1}(K).
\]
Since $\psi^{-1}(K)$ is compact in $\mathbb{D}$, there exists
$r\in(0,1)$ such that
\[
\mathbb{D}\setminus\overline{\mathbb{D}_{r}}
\subset h(\mathbb{D}).
\]
By Lemma~\ref{lem2.1},
\[
N_{h,\alpha}(w)\geq(1-|w|^{2})^{\alpha},
\qquad w\in\mathbb{D}\setminus\overline{\mathbb{D}_{r}}.
\]
For $f\in\mathcal{D}_{\alpha}$, let $F=f\circ\psi$, initially regarded
only as an analytic function. The area formula for $h$, with nonnegative
integrals allowed to be infinite, and Lemma~\ref{lem2.8} give
\begin{align*}
\|C_{\varphi}f\|_{\mathcal{D}_{\alpha}}^{2}
&=|f(0)|^{2}+(1+\alpha)
  \int_{\mathbb{D}}|F'(w)|^{2}N_{h,\alpha}(w)\,dA(w)\\
&\geq |f(0)|^{2}+
  \int_{\mathbb{D}\setminus\overline{\mathbb{D}_{r}}}
  |F'(w)|^{2}\,dA_{\alpha}(w)\\
&\geq (1-r^{2})^{\alpha+1}
  \left(|F(0)|^{2}+
  \int_{\mathbb{D}}|F'(w)|^{2}\,dA_{\alpha}(w)\right).
\end{align*}
The left-hand side is finite. Hence, $F\in\mathcal{D}_{\alpha}$ and
\[
\|C_{\psi}f\|_{\mathcal{D}_{\alpha}}^{2}
\leq (1-r^{2})^{-\alpha-1}\|C_{\varphi}\|^{2}
\|f\|_{\mathcal{D}_{\alpha}}^{2}.
\]
This proves the boundedness of $C_{\psi}$ without assuming that
$C_h$ is bounded.
\end{proof}

\begin{lemma}
\label{lem2.13} Let $-1<\alpha<0$, and let $\varphi$ be a normalized
self-map of $\mathbb{D}$ in $\mathcal{D}_{\alpha}$. Suppose that
$\Omega=\varphi(\mathbb{D})$ satisfies the Compact-Hole Condition. Then
\[
\mathbb{T}\subset\overline{\Omega}
\]
if and only if there exists $r_{0}\in(0,1)$ such that
\[
\mathbb{D}\setminus\overline{\mathbb{D}_{r_{0}}}\subset\Omega.
\]
In this case $\Omega=\mathbb{D}\setminus K$ for some compact
$K\subset\mathbb{D}$. If $\Omega$ is simply connected, then
$\Omega=\mathbb{D}$.
\end{lemma}

\begin{proof}
Only the forward implication requires proof. Choose a decomposition
$\Omega=U\setminus K$ as in Definition~\ref{def2.7}, and a
conformal bijection $\psi$ of $\mathbb{D}$ onto $U$ with $\psi(0)=0$.
Lemma~\ref{lem2.12} shows that $C_{\psi}$ is bounded on
$\mathcal{D}_{\alpha}$. Since $\Omega\subset U$, the hypothesis gives
$\mathbb{T}\subset\overline{U}$. Lemma~\ref{lem2.9} therefore implies
that $\psi$ is an automorphism, so $U=\mathbb{D}$.
Consequently $K\Subset\mathbb{D}$, and we may choose $r_{0}\in(0,1)$
such that $K\subset\mathbb{D}_{r_{0}}$. This gives the required
annular inclusion. If $\Omega$ is simply connected, we may take
$U=\Omega$ and $K=\emptyset$ in the same argument, giving $\Omega=\mathbb{D}$.
\end{proof}

We can now characterize closed range for bounded composition operators
whose symbols have images satisfying the Compact-Hole Condition.

\begin{theorem}
\label{thm2.14} Let $-1<\alpha<0$, and let $\varphi$ be a normalized
self-map of $\mathbb{D}$ in $\mathcal{D}_{\alpha}$. Suppose that
$\varphi(\mathbb{D})$ satisfies the Compact-Hole Condition. Then the
following statements are equivalent:

\begin{enumerate}[label=\textup{(\roman*)}]
\item The range of $C_{\varphi}$ is closed in $\mathcal{D}_{\alpha}$.

\item $\mu_{\varphi,\alpha}$ satisfies the Reverse Carleson Disk Condition;
that is, there exist constants $\delta>0$ and $r>0$ such that
\[
\mu_{\varphi,\alpha}(D(z,r))\geq\delta A_{\alpha}(D(z,r)),
\qquad z\in\mathbb{D}.
\]

\item There exist constants $\delta>0$ and $r>0$ such that
\[
A_{\alpha}\bigl(\varphi(\mathbb{D})\cap D(z,r)\bigr)
\geq\delta A_{\alpha}\bigl(D(z,r)\bigr),\qquad z\in\mathbb{D}.
\]

\item There exists $r_{0}\in(0,1)$ such that
\[
\mathbb{D}\setminus\overline{\mathbb{D}_{r_{0}}}
\subset\varphi(\mathbb{D}).
\]

\item $\mu_{\varphi,\alpha}$ is a Reverse Carleson measure for
$A_{\alpha}^{2}$.
\end{enumerate}
Moreover, these conditions are equivalent to
$\mathbb{T}\subset\overline{\varphi(\mathbb{D})}$.
In particular, when $\varphi(\mathbb{D})$ is simply connected,
they are equivalent to $\varphi(\mathbb{D})=\mathbb{D}$.
\end{theorem}

\begin{proof}
The implication \textup{(i)}$\Rightarrow$\textup{(ii)} follows from
Theorem~\ref{thm2.4}. Assume \textup{(ii)}. If
$\zeta\in\mathbb{T}\setminus\overline{\varphi(\mathbb{D})}$, some
Euclidean neighborhood of $\zeta$ misses $\varphi(\mathbb{D})$.
For $t<1$ sufficiently close to $1$, the disk $D(t\zeta,r)$ is contained
in this neighborhood, where $r$ is the fixed radius in \textup{(ii)}. Then $\mu_{\varphi,\alpha}(D(t\zeta,r))=0$, a contradiction.
Thus $\mathbb{T}\subset\overline{\varphi(\mathbb{D})}$.
Lemma~\ref{lem2.13} gives
\[
\mathbb{D}\setminus\overline{\mathbb{D}_{r_{0}}}
\subset\varphi(\mathbb{D})
\]
for some $r_{0}\in(0,1)$, which is \textup{(iv)}.

Assume \textup{(iv)}. By Lemma~\ref{lem2.8}, for every $g\in A_{\alpha}^{2}$,
\[
\int_{\varphi(\mathbb{D})}|g(w)|^{2}\,dA_{\alpha}(w)
\geq (1-r_{0}^{2})^{\alpha+1}\|g\|_{A_{\alpha}^{2}}^{2}.
\]
Lemma~\ref{lem2.10} therefore yields \textup{(iii)}.
Finally, Lemma~\ref{lem2.11} gives
\textup{(iii)}$\Rightarrow$\textup{(v)}, and
Proposition~\ref{prop2.2} gives \textup{(v)}$\Rightarrow$\textup{(i)}.
This proves the equivalence of the five conditions. The equivalence with
$\mathbb{T}\subset\overline{\varphi(\mathbb{D})}$ and the simply connected
case follow from Lemma~\ref{lem2.13}.
\end{proof}

Since every finitely connected domain satisfies the Compact-Hole
Condition, Theorem~\ref{thm2.14} applies in particular to all normalized
self-maps in $\mathcal{D}_{\alpha}$ with finitely connected image.
The geometric assumption concerns the image of $\varphi$ and does not
require $\varphi$ to be univalent. In particular, the power maps in
Example~\ref{ex2.17} satisfy the simply connected case of the theorem.

\begin{corollary}
\label{cor2.15} Let $-1<\alpha<0$, and let $\Phi$ be a nonconstant
analytic self-map of $\mathbb{D}$ such that $C_{\Phi}$ is bounded on
$\mathcal{D}_{\alpha}$ and $\Phi(\mathbb{D})$ satisfies the Compact-Hole
Condition.
Set $\varphi=\sigma_{\Phi(0)}\circ\Phi$. Then the range of $C_{\Phi}$
is closed in $\mathcal{D}_{\alpha}$ if and only if the equivalent
conditions \textup{(ii)--(v)} in Theorem~\ref{thm2.14} hold for
$\varphi$. Equivalently,
\[
\mathbb{T}\subset\overline{\Phi(\mathbb{D})}.
\]
\end{corollary}

\begin{proof}
The operator $C_{\sigma_{\Phi(0)}}$ is bounded and invertible, and
$C_{\Phi}=C_{\varphi}C_{\sigma_{\Phi(0)}}$. The disk automorphism
$\sigma_{\Phi(0)}$ preserves the Compact-Hole Condition: if
$\Phi(\mathbb{D})=U\setminus K$, then
$\varphi(\mathbb{D})=\sigma_{\Phi(0)}(U)\setminus\sigma_{\Phi(0)}(K)$,
where $\sigma_{\Phi(0)}(U)$ is simply connected and
$\sigma_{\Phi(0)}(K)\Subset\sigma_{\Phi(0)}(U)$.
It also preserves contact with the entire unit circle. The conclusion
follows from Theorem~\ref{thm2.14}.
\end{proof}

The next two examples satisfy the Compact-Hole Condition and show
that the closed range property can fail or hold even for simply
connected images.

\begin{example}
\label{ex2.16} The Compact-Hole Condition and boundedness do not by
themselves imply closed range. Let $-1<\alpha<0$, $0<a<1$, and
$\varphi_a(z)=az$. Its image $a\mathbb{D}$ is simply connected.
The monomials are orthogonal in $\mathcal{D}_{\alpha}$ and
$C_{\varphi_a}z^n=a^nz^n$, so $C_{\varphi_a}$ is bounded. However, for
$f_n=z^n/\|z^n\|_{\mathcal{D}_{\alpha}}$, $n\geq1$, we have
\[
\|f_n\|_{\mathcal{D}_{\alpha}}=1,\qquad
\|C_{\varphi_a}f_n\|_{\mathcal{D}_{\alpha}}=a^n\rightarrow0.
\]
Thus $C_{\varphi_a}$ is not bounded below and its range is not closed.
Also, $\overline{\varphi_a(\mathbb{D})}$ misses $\mathbb{T}$, and every
fixed-radius Bergman disk sufficiently near $\mathbb{T}$ misses the
image. Hence the Reverse Carleson Disk Condition and the area density
condition fail as well.
\end{example}

\begin{example}
\label{ex2.17} Let $\alpha>-1$ and let
\[
\varphi(z)=z^{m}, \qquad m\in\mathbb{N}.
\]
Then $\varphi$ is a nonconstant analytic self-map of $\mathbb{D}$ satisfying
$\varphi(0)=0$. For $w\in\mathbb{D}$, counting preimages with multiplicity,
\[
N_{\varphi,\alpha}(w) = m\bigl(1-|w|^{2/m}\bigr)^{\alpha}.
\]
Put $t=|w|^{2/m}$. Since
\[
1-|w|^{2} = 1-t^{m} = (1-t)(1+t+\cdots+t^{m-1})
\]
and
\[
1\leq1+t+\cdots+t^{m-1}\leq m,
\]
there exist constants $0<c_{m,\alpha}\leq C_{m,\alpha}<\infty$ such that
\[
c_{m,\alpha}(1-|w|^{2})^{\alpha}\leq N_{\varphi,\alpha}(w) \leq C_{m,\alpha
}(1-|w|^{2})^{\alpha}, \qquad w\in\mathbb{D}.
\]
Consequently,
\[
\frac{c_{m,\alpha}}{1+\alpha} \leq\widetilde\mu_{\varphi,\alpha,r}(z)
\leq\frac{C_{m,\alpha}}{1+\alpha}, \qquad z\in\mathbb{D},\quad r>0.
\]
By the change of variables formula, these bounds imply
\[
\|C_{\varphi}f\|_{\mathcal{D}_{\alpha}}^2
\asymp\|f\|_{\mathcal{D}_{\alpha}}^2,
\qquad f\in\mathcal{D}_{\alpha}.
\]
Thus $C_{\varphi}$ is bounded and has closed range. Since
$\varphi(\mathbb{D})=\mathbb{D}$, its image is simply connected. In
particular, for $-1<\alpha<0$ this gives a class for which all five
conditions in Theorem~\ref{thm2.14} hold.
\end{example}

\begin{example}
\label{ex2.18} The boundedness assumption of $C_{\varphi}$ in Lemma~\ref{lem2.13}
cannot be omitted, even when the image is simply connected. Let
$-1<\alpha<0$ and define
\[
H = \left\{  w\in\mathbb{D}: \left|  w-\frac12\right|  \leq\frac18 \right\}  ,
\qquad\Gamma=\left[  \frac58,1\right)  ,
\]
and set
\[
\Omega=\mathbb{D}\setminus(H\cup\Gamma).
\]

\begin{tikzpicture}[
x=3.0cm,
y=3.0cm,
every node/.style={font=\small}
]
% Unit disk
\fill[gray!4] (0,0) circle[radius=1];
\draw[thick] (0,0) circle[radius=1];
% Real axis
\draw[gray!60] (-1.05,0)--(1.05,0);
% Removed disk H
\fill[gray!30] (0.5,0) circle[radius=0.125];
\draw[thick] (0.5,0) circle[radius=0.125];
% Removed slit Gamma
\draw[line width=1.1pt] (0.625,0)--(0.995,0);
% Points
\fill (0,0) circle[radius=0.009];
\fill (0.5,0) circle[radius=0.009];
% Labels
\node at (-0.30,0.40) {$\Omega$};
\node[above=3pt] at (0.50,0.125) {$H$};
\node[above=3pt] at (0.80,0) {$\Gamma$};
% Axis labels
\node[left=2pt] at (-1,0) {$-1$};
\node[right=2pt] at (1,0) {$1$};
% Number labels
\node[font=\tiny, below=8pt] at (0,0) {$0$};
% inside H: slightly to the right
\node[font=\fontsize{4}{4}\selectfont] at (0.525,-0.055)
{$\frac12$};
% outside H: move slightly left
\node[font=\fontsize{4}{4}\selectfont] at (0.675,-0.055)
{$\frac58$};
\end{tikzpicture}

\smallskip

{\small The domain $\Omega=\mathbb{D}-\left(  H\cup\Gamma\right)  $, where $H$
is the shaded disk and $\Gamma$ is the radial slit.}

Since $\Gamma$ joins $H$ to the boundary of $\mathbb{D}$, the domain $\Omega$
is simply connected, and $0\in\Omega$. By the Riemann mapping theorem there
exists a conformal map
\[
\varphi:\mathbb{D}\longrightarrow\Omega
\]
with $\varphi(0)=0$. Thus $\varphi$ is a univalent analytic self-map of
$\mathbb{D}$ and $\varphi(\mathbb{D})=\Omega$.

We first show that the Reverse Carleson Disk Condition holds. Set
\[
M_{H}=\max_{\zeta\in H}\beta(0,\zeta)<\infty
\]
and choose $R_{0}>M_{H}$. For $z\in\mathbb{D}$, define
\[
\Theta(z) = \frac{A_{\alpha}(D(z,R_{0})\setminus H)}{A_{\alpha}(D(z,R_{0}))}.
\]
We claim that $\Theta(z)>0$ for every $z\in\mathbb{D}$. If $z\notin H$, then
$D(z,R_{0})\setminus H$ contains a neighborhood of $z$. If $z\in H$, then
$\beta(z,0)\leq M_{H}<R_{0}$, so that $0\in D(z,R_{0})\setminus H$; hence
$D(z,R_{0})\setminus H$ has positive weighted area. For any sequence $z_n\to z$ in $\mathbb{D}$, we have $\chi_{D(z_n,R_0)}\to\chi_{D(z,R_0)}$ outside $\partial D(z,R_0)$, which has zero $A_{\alpha}$-measure. Since the characteristic functions are bounded by $1$ and $A_{\alpha}(\mathbb{D})<\infty$, the dominated convergence theorem implies that $\Theta$ is continuous. Moreover, if $\beta(0,z)>M_{H}+R_{0}$, then $D(z,R_{0})\cap
H=\emptyset$, and hence $\Theta(z)=1$. Thus $\Theta$ need only be considered
on the compact Bergman disk
\[
\{z\in\mathbb{D}:\beta(0,z)\leq M_{H}+R_{0}\}.
\]

It follows that there exists $c>0$ such that
\[
A_{\alpha}(D(z,R_{0})\setminus H) \geq cA_{\alpha}(D(z,R_{0})), \qquad
z\in\mathbb{D}.
\]
Since $\Gamma$ has planar area zero,
\[
A_{\alpha}(D(z,R_{0})\cap\Omega) = A_{\alpha}(D(z,R_{0})\setminus H).
\]

For $w\in\Omega$, let $\zeta=\varphi^{-1}(w)$. Since $\varphi(0)=0$, Schwarz's
lemma gives
\[
|w|=|\varphi(\zeta)|\leq|\zeta|.
\]
As $\varphi$ is univalent,
\[
N_{\varphi,\alpha}(w)=(1-|\zeta|^{2})^{\alpha}.
\]
Since $-1<\alpha<0$, we obtain
\[
N_{\varphi,\alpha}(w) \geq(1-|w|^{2})^{\alpha}, \qquad w\in\Omega.
\]
Therefore,
\begin{align*}
\mu_{\varphi,\alpha}(D(z,R_{0}))  &  = \int_{D(z,R_{0})\cap\Omega}%
N_{\varphi,\alpha}(w)\,dA(w)\\
&  \geq\int_{D(z,R_{0})\cap\Omega}(1-|w|^{2})^{\alpha}\,dA(w)\\
&  = \frac{1}{1+\alpha}A_{\alpha}(D(z,R_{0})\cap\Omega)\\
&  \geq\frac{c}{1+\alpha}A_{\alpha}(D(z,R_{0})).
\end{align*}
Thus $\mu_{\varphi,\alpha}$ satisfies the Reverse Carleson Disk
Condition, and $\Omega$ has uniform weighted area density at the radius
$R_0$.

The construction also gives
\[
\overline{\Omega}=\overline{\mathbb{D}}\setminus\operatorname{int}H.
\]
In particular, $\mathbb{T}\subset\overline{\Omega}$, and
\[
\mathbb{D}\setminus\overline{\mathbb{D}_{r_0}}\subset\overline{\Omega},
\qquad r_0\in(5/8,1).
\]
However, no outer annulus
is contained in $\Omega$, because the slit $\Gamma$ reaches the unit circle.
Moreover, $\Omega$ is not dense in $\mathbb{D}$, since it omits the interior
of $H$.

If $C_{\varphi}$ were bounded on $\mathcal{D}_{\alpha}$, the univalence of
$\varphi$ and the inclusion $\mathbb{T}\subset\overline{\varphi(\mathbb{D})}$
would imply, by Lemma~\ref{lem2.9}, that $\varphi$ is an automorphism. This
contradicts $\varphi(\mathbb{D})=\Omega\ne\mathbb{D}$. Hence
$C_{\varphi}$ is unbounded. Since $\Omega$ is simply connected, it
satisfies the Compact-Hole Condition with $U=\Omega$ and $K=\emptyset$.
The example therefore shows that the Compact-Hole Condition and the
Reverse Carleson Disk Condition do not ensure boundedness, and that
boundedness is essential for the actual annulus inclusion in
Lemma~\ref{lem2.13}.
\end{example}

\subsection{Uniform Tail Condition of $\varphi$ and closed range of
$C_{\varphi}$}\mbox{}

We now obtain a characterization without imposing the Compact-Hole
Condition on the image. The Uniform Tail Condition controls the contribution of large values of
$N_{\varphi,\alpha}(w)/(1-|w|^{2})^{\alpha}$ uniformly over the unit sphere
of $\mathcal{D}_{\alpha,0}$.

Under this condition, the Reverse Carleson Disk Condition yields a weighted
area density estimate for $\varphi(\mathbb{D})$, leading to another
characterization of closed range of $C_{\varphi}$.

\begin{definition}[Uniform Tail Condition]
\label{def2.19}
Let $\alpha>-1$, and let $\varphi$ be a nonconstant analytic self-map of
$\mathbb{D}$ satisfying $\varphi(0)=0$. For $M>0$, define
\[
B_{M}=\left\{w\in\varphi(\mathbb{D}):
0<\frac{N_{\varphi,\alpha}(w)}{(1-|w|^{2})^{\alpha}}\leq M\right\}
\]
and
\[
E_{M}=\left\{w\in\varphi(\mathbb{D}):
\frac{N_{\varphi,\alpha}(w)}{(1-|w|^{2})^{\alpha}}>M\right\}.
\]
We say that $\varphi$ satisfies the Uniform Tail Condition if
\begin{equation}
\label{eq:UTC}
\lim_{M\to\infty}
\sup_{\substack{f\in\mathcal{D}_{\alpha,0}\\
\|f\|_{\mathcal{D}_{\alpha}}=1}}
\int_{E_{M}}|f^{\prime}(w)|^{2}
N_{\varphi,\alpha}(w)\,dA(w)=0.
\tag{UTC}
\end{equation}
The integrals and the supremum are understood in $[0,+\infty]$.
The boundedness of $C_{\varphi}$ follows from this condition by
Lemma~\ref{lem2.21}.
\end{definition}

\begin{lemma}
\label{lem2.20} Let $\alpha>-1$, and let $\varphi$ be a nonconstant analytic
self-map of $\mathbb{D}$ satisfying $\varphi(0)=0$. Suppose that there are
constants $\delta>0$ and $r>0$ such that
\[
\mu_{\varphi,\alpha}(D(z,r)) \geq\delta A_{\alpha}(D(z,r)), \qquad
z\in\mathbb{D},
\]
and assume that $\varphi$ satisfies the Uniform Tail Condition
\eqref{eq:UTC}. Then there exists $C>0$ such that
\[
A_{\alpha}(\varphi(\mathbb{D})\cap D(z,r)) \geq C A_{\alpha}(D(z,r)), \qquad
z\in\mathbb{D}.
\]

\end{lemma}

\begin{proof}
Let
\[
\kappa_{z}^{\alpha}(w) = \frac{(1-|z|^{2})^{(2+\alpha)/2}} {(1-\overline
zw)^{2+\alpha}}
\]
be the normalized reproducing kernel of $A_{\alpha}^{2}$, and put
\[
F_{z}(\xi)=\int_{0}^{\xi}\kappa_{z}^{\alpha}(w)\,dw.
\]
Then $F_{z}\in\mathcal{D}_{\alpha,0}$ and $\|F_{z}\|_{\mathcal{D}_{\alpha}}%
=1$. Since $r$ is fixed, there are constants $c_{1},c_{2}>0$, independent of
$z$, such that
\[
\frac{c_{1}}{A_{\alpha}(D(z,r))} \leq|\kappa_{z}^{\alpha}(w)|^{2} \leq
\frac{c_{2}}{A_{\alpha}(D(z,r))}, \qquad w\in D(z,r).
\]
Consequently, there exists a constant $C_{1}>0$ independent of $z$ and $M$ such that
\begin{align*}
\delta &  \leq\frac{1}{A_{\alpha}(D(z,r))} \int_{D(z,r)\cap\varphi
(\mathbb{D})} N_{\varphi,\alpha}(w)\,dA(w)\\
&  \leq C_{1}\int_{D(z,r)\cap\varphi(\mathbb{D})} |\kappa_{z}^{\alpha}%
(w)|^{2}N_{\varphi,\alpha}(w)\,dA(w)\\
&  \leq C_{1}\sup_{\substack{f\in\mathcal{D}_{\alpha,0}\\\|f\|_{\mathcal{D}%
_{\alpha}}=1}} \int_{E_{M}}|f^{\prime}(w)|^{2}N_{\varphi,\alpha}(w)\,dA(w)\\
&  \quad+ \frac{C_{1}c_{2}}{A_{\alpha}(D(z,r))} \int_{D(z,r)\cap B_{M}%
}N_{\varphi,\alpha}(w)\,dA(w).
\end{align*}
Choose $M_{0}>0$ so large that the first term in the last expression is at
most $\delta/2$. For $w\in B_{M_{0}}$,
\[
N_{\varphi,\alpha}(w) \leq M_{0}(1-|w|^{2})^{\alpha}.
\]
Hence
\[
\int_{D(z,r)\cap B_{M_{0}}}N_{\varphi,\alpha}(w)\,dA(w) \leq\frac{M_{0}%
}{1+\alpha} A_{\alpha}(D(z,r)\cap\varphi(\mathbb{D})).
\]
It follows that
\[
A_{\alpha}(D(z,r)\cap\varphi(\mathbb{D})) \geq\frac{(1+\alpha)\delta}%
{2C_{1}c_{2}M_{0}} A_{\alpha}(D(z,r)), \qquad z\in\mathbb{D}.
\]
This proves the assertion.
\end{proof}

\begin{lemma}
\label{lem2.21} Let $\alpha>-1$, and let $\varphi$ be a nonconstant analytic
self-map of $\mathbb{D}$ satisfying $\varphi(0)=0$. If the Uniform Tail
Condition holds, then $C_{\varphi}$ is bounded on $\mathcal{D}_{\alpha}$.
\end{lemma}

\begin{proof}
Choose $M_{0}>0$ such that
\[
\sup_{\substack{f\in\mathcal{D}_{\alpha,0}\\\|f\|_{\mathcal{D}_{\alpha}}=1}}
\int_{E_{M_{0}}} |f^{\prime}(w)|^{2}N_{\varphi,\alpha}(w)\,dA(w) \leq1.
\]
For $f\in\mathcal{D}_{\alpha,0}$, we have
\begin{align*}
\int_{\mathbb{D}}|f^{\prime}(w)|^{2}N_{\varphi,\alpha}(w)\,dA(w)  &  = \int
_{B_{M_{0}}}|f^{\prime}(w)|^{2}N_{\varphi,\alpha}(w)\,dA(w)\\
&  \quad+ \int_{E_{M_{0}}}|f^{\prime}(w)|^{2}N_{\varphi,\alpha}(w)\,dA(w)\\
&  \leq\frac{M_{0}}{1+\alpha} \|f\|_{\mathcal{D}_{\alpha}}^{2} +
\|f\|_{\mathcal{D}_{\alpha}}^{2}.
\end{align*}
Thus $C_{\varphi,0}$ is bounded on $\mathcal{D}_{\alpha,0}$. Since
\[
C_{\varphi}= I_{\mathbb{C}}\oplus C_{\varphi,0},
\]
the operator $C_{\varphi}$ is bounded on $\mathcal{D}_{\alpha}$.
\end{proof}

\begin{theorem}
\label{thm2.22} Let $-1<\alpha<0$, and let $\varphi$ be a nonconstant analytic
self-map of $\mathbb{D}$ satisfying $\varphi(0)=0$. Suppose that $\varphi$
satisfies the Uniform Tail Condition. Then the following statements are equivalent:

\begin{enumerate}
[label=\textup{(\roman*)}]

\item $C_{\varphi}$ has closed range on $\mathcal{D}_{\alpha}$.

\item $\mu_{\varphi,\alpha}$ satisfies the Reverse Carleson Disk Condition;
that is, there exist constants $\delta>0$ and $r>0$ such that
\[
\mu_{\varphi,\alpha}(D(z,r)) \geq\delta A_{\alpha}(D(z,r)), \qquad
z\in\mathbb{D}.
\]

\item There exist constants $\delta>0$ and $r>0$ such that
\[
A_{\alpha}(\varphi(\mathbb{D})\cap D(z,r)) \geq\delta A_{\alpha}(D(z,r)),
\qquad z\in\mathbb{D}.
\]

\item $\mu_{\varphi,\alpha}$ is a Reverse Carleson measure for $A_{\alpha}%
^{2}$.
\end{enumerate}
\end{theorem}

\begin{proof}
By Lemma~\ref{lem2.21}, $C_{\varphi}$ is bounded on $\mathcal{D}_{\alpha}$.

Theorem~\ref{thm2.4} gives \textup{(i)}$\Rightarrow$\textup{(ii)}. Under the
Uniform Tail hypothesis, Lemma~\ref{lem2.20} gives \textup{(ii)}$\Rightarrow
$\textup{(iii)}. Lemma~\ref{lem2.11} gives \textup{(iii)}$\Rightarrow
$\textup{(iv)}, while Proposition~\ref{prop2.2} gives \textup{(iv)}%
$\Rightarrow$\textup{(i)}. Therefore all four conditions are equivalent.
\end{proof}
The proofs of Lemmas~\ref{lem2.20} and~\ref{lem2.21} require only
$\alpha>-1$, whereas the counting-function estimate used in
Lemma~\ref{lem2.11} relies on $\alpha\leq0$. Here we retain the range
$-1<\alpha<0$ to focus on singular weights.

The Uniform Tail hypothesis is genuinely nontrivial: the normalized counting
function may be unbounded while the contribution of its high level sets still
tends uniformly to zero. We return to the dilation in
Example~\ref{ex2.16} to illustrate this point.

\begin{example}
\label{ex2.23} Let $-1<\alpha<0$, fix $a\in(0,1)$, and consider
\[
\varphi_{a}(z)=az, \qquad z\in\mathbb{D}.
\]
We show that $\varphi_{a}$ satisfies the Uniform Tail Condition. Since
$\varphi_{a}$ is univalent and $\varphi_{a}(\mathbb{D})=a\mathbb{D}$,
\[
N_{\varphi_{a},\alpha}(w) = \left(  1-\frac{|w|^{2}}{a^{2}}\right)  ^{\alpha},
\qquad|w|<a.
\]
Hence
\[
\frac{N_{\varphi_{a},\alpha}(w)}{(1-|w|^{2})^{\alpha}} = \left(
\frac{1-|w|^{2}/a^{2}}{1-|w|^{2}} \right)  ^{\alpha}.
\]
For $M>1$, put $\varepsilon_{M}=M^{1/\alpha}$. Since $\alpha<0$, we have $\varepsilon_{M}\to0$ as $M\to\infty$. A direct calculation gives
\[
E_{M} = \{w\in\mathbb{D}:r_{M}<|w|<a\},
\]
where
\[
r_{M}^{2} = a^{2}\frac{1-\varepsilon_{M}}{1-a^{2}\varepsilon_{M}}.
\]
In particular, $r_{M}\to a$ as $M\to\infty$.

Let $f\in\mathcal{D}_{\alpha,0}$ with $\|f\|_{\mathcal{D}_{\alpha}}=1$, and
put $g=f^{\prime}$. Then $\|g\|_{A_{\alpha}^{2}}=1$. By the point evaluation
estimate for $A_{\alpha}^{2}$,
\[
|f^{\prime}(w)|^{2}=|g(w)|^{2} \leq\frac{1}{(1-|w|^{2})^{\alpha+2}} \leq\frac
{1}{(1-a^{2})^{\alpha+2}}, \qquad|w|\leq a.
\]
Therefore,
\begin{align*}
 \sup_{\substack{f\in\mathcal{D}_{\alpha,0}\\\|f\|_{\mathcal{D}_{\alpha}}%
=1}} \int_{E_{M}}|f^{\prime}(w)|^{2}N_{\varphi_{a},\alpha}(w)\,dA(w)
&\lesssim\int_{E_{M}}N_{\varphi_{a},\alpha}(w)\,dA(w)\\
\qquad&= 2\int_{r_{M}}^{a} \left(  1-\frac{r^{2}}{a^{2}}\right)  ^{\alpha
}r\,dr\\
\qquad&=\frac{a^{2}}{\alpha+1} \left(  1-\frac{r_{M}^{2}}{a^{2}}\right)
^{\alpha+1}\\
\qquad&=\frac{a^{2}}{\alpha+1} \left(  \frac{(1-a^{2})\varepsilon_{M}%
}{1-a^{2}\varepsilon_{M}} \right)  ^{\alpha+1}\\
\qquad&\lesssim M^{(\alpha+1)/\alpha}.
\end{align*}
Since $\frac{\alpha+1}{\alpha}<0$, the right-hand side tends to zero as $M\to\infty$. Thus $\varphi_{a}$
satisfies the Uniform Tail Condition.

This example also clarifies the role of the hypothesis in
Theorem~\ref{thm2.22}. Indeed,
\[
C_{\varphi_{a}}z^{n}=a^{n}z^{n},
\]
so $C_{\varphi_{a}}$ is not bounded below and hence does not have closed
range. At the same time, $\varphi_{a}(\mathbb{D})=a\mathbb{D}$ fails the
uniform area density condition near the boundary. For every fixed $r>0$, if
$|z|$ is sufficiently close to $1$, then $D(z,r)\cap a\mathbb{D}=\emptyset$,
and hence
\[
\mu_{\varphi_{a},\alpha}(D(z,r))=0.
\]
Thus the Reverse Carleson Disk Condition in Theorem~\ref{thm2.22} also fails.
Therefore, the Uniform Tail Condition is a background hypothesis for the
equivalence in Theorem~\ref{thm2.22}, rather than a closed range criterion by itself.
\end{example}

\begin{remark}
\label{re2.24} The preceding example shows that the Uniform Tail Condition
alone does not imply closed range. In contrast, the power maps considered in
Example~\ref{ex2.17} provide a simple class for which the uniform tail
condition and all the equivalent conditions in Theorem~\ref{thm2.22} hold.

Indeed, let
\[
\varphi(z)=z^{m}, \qquad m\geq2.
\]
For $w\in\mathbb{D}$,
\[
N_{\varphi,\alpha}(w)=m\bigl(1-|w|^{2/m}\bigr)^{\alpha}.
\]
Writing $t=|w|^{2/m}$, we obtain
\[
\frac{N_{\varphi,\alpha}(w)} {(1-|w|^{2})^{\alpha}}=m(1+t+\cdots
+t^{m-1})^{-\alpha}.
\]
Since $1\leq1+t+\cdots+t^{m-1}\leq m$, it follows that
\[
m \leq\frac{N_{\varphi,\alpha}(w)} {(1-|w|^{2})^{\alpha}} \leq m^{1-\alpha}.
\]
Consequently,
\[
E_{M}=\emptyset\qquad\text{whenever }M>m^{1-\alpha},
\]
so the Uniform Tail Condition holds trivially.

Moreover, $\varphi(\mathbb{D})=\mathbb{D}$, and hence
\[
A_{\alpha}\bigl(\varphi(\mathbb{D})\cap D(z,r)\bigr)= A_{\alpha}(D(z,r)),
\qquad z\in\mathbb{D}.
\]
Also,
\[
N_{\varphi,\alpha}(w) \geq m(1-|w|^{2})^{\alpha}.
\]
Hence, for every $r>0$,
\[
\mu_{\varphi,\alpha}(D(z,r)) \geq\frac{m}{1+\alpha}A_{\alpha}(D(z,r)), \qquad
z\in\mathbb{D},
\]
so the Reverse Carleson Disk Condition in Theorem~\ref{thm2.22} holds (in
fact, at every radius). The same estimate shows that $\mu_{\varphi,\alpha}$ is
a Reverse Carleson measure for $A_{\alpha}^{2}$. Thus $C_{\varphi}$ has closed
range by Proposition~\ref{prop2.2}.
\end{remark}

\begin{corollary}
\label{cor2.25}
Let $-1<\alpha<0$, and let $\Phi$ be a nonconstant analytic self-map
of $\mathbb{D}$. Set
\[
\varphi=\sigma_{\Phi(0)}\circ\Phi.
\]
Suppose that $\varphi$ satisfies the Uniform Tail Condition.
Then the range of $C_{\Phi}$ is closed in $\mathcal{D}_{\alpha}$
if and only if the equivalent conditions \textup{(ii)--(iv)} in
Theorem~\ref{thm2.22} hold for $\varphi$.
\end{corollary}

\begin{proof}
Since $C_{\sigma_{\Phi(0)}}$ is bounded and invertible on
$\mathcal{D}_{\alpha}$ and
$C_{\Phi}=C_{\varphi}C_{\sigma_{\Phi(0)}}$, the conclusion follows
from Theorem~\ref{thm2.22}.
\end{proof}
\section{Essential norms of $C_{\varphi}$}
In this section, $\mathcal{D}_{\alpha}$ continues to denote the weighted
Dirichlet space, while $\mathcal{D}_{\alpha,0}$ denotes the subspace of functions
vanishing at the origin. We first transfer the essential norm problem
to positive Toeplitz operators on $A_\alpha^2$ and obtain a general
local average estimate and an exact boundary tail formula. We then
derive exact local average formulas under a boundary oscillation
condition on the counting density.

For a bounded operator $T:H_1\to H_2$ between Hilbert spaces, its
essential norm is defined by
\[
\|T\|_e=\inf\{\|T-K\|:K:H_1\to H_2\text{ is compact}\}.
\]
We denote the inner product on $A_{\alpha}^{2}$ by
\[
\langle f,g\rangle_{\alpha}
=\int_{\mathbb{D}}f(w)\overline{g(w)}\,dA_{\alpha}(w),
\qquad f,g\in A_{\alpha}^{2}.
\]
\begin{definition}
\label{def3.1} Let $\alpha>-1$. Define
\[
U_{\alpha}:\mathcal{D}_{\alpha,0}\longrightarrow A_{\alpha}^{2}, \qquad
U_{\alpha}f=f^{\prime}.
\]
For an analytic self-map $\varphi$ of $\mathbb{D}$, define
\[
\widetilde C_{\varphi}= P_{0}C_{\varphi}\big|_{\mathcal{D}_{\alpha,0}},
\qquad\widetilde C_{\varphi}f = f\circ\varphi-f(\varphi(0)),
\]
and
\[
D_{\varphi}g = (g\circ\varphi)\varphi^{\prime}, \qquad g\in A_{\alpha}^{2}.
\]

\end{definition}

\begin{proposition}
\label{prop3.2} Let $\alpha>-1$, and let $\varphi:\mathbb{D}\to\mathbb{D}$ be
a nonconstant analytic self-map. Then
\[
C_{\varphi}\text{ is bounded on }\mathcal{D}_{\alpha}\quad\Longleftrightarrow
\quad\widetilde C_{\varphi}\text{ is bounded on }\mathcal{D}_{\alpha,0}
\quad\Longleftrightarrow\quad D_{\varphi}\text{ is bounded on }A_{\alpha}%
^{2}.
\]
Whenever these conditions hold,
\[
U_{\alpha}\widetilde C_{\varphi}U_{\alpha}^{*}=D_{\varphi}%
\]
and
\[
\|C_{\varphi}\|_{e} = \|\widetilde C_{\varphi}\|_{e} = \|D_{\varphi}\|_{e}.
\]

\end{proposition}

\begin{proof}
The operator $U_{\alpha}$ is unitary because
\[
\|f\|_{\mathcal{D}_{\alpha}}^{2} = \int_{\mathbb{D}}|f^{\prime}(z)|^{2}\,dA_{\alpha
}(z), \qquad f\in\mathcal{D}_{\alpha,0},
\]
and
\[
(U_{\alpha}^{*}g)(z)=\int_{0}^{z} g(\zeta)\,d\zeta.
\]
For $g\in A_{\alpha}^{2}$,
\begin{align*}
(U_{\alpha}\widetilde C_{\varphi}U_{\alpha}^{*}g)(z)  &  = \frac{d}{dz}
\left[  (U_{\alpha}^{*}g)(\varphi(z))-(U_{\alpha}^{*}g)(\varphi(0))\right] \\
&  = g(\varphi(z))\varphi^{\prime}(z) = D_{\varphi}g(z).
\end{align*}
Thus $\widetilde C_{\varphi}$ is bounded if and only if $D_{\varphi}$ is bounded.

Relative to $\mathcal{D}_{\alpha}=\mathbb{C}\oplus\mathcal{D}_{\alpha,0}$, we
have
\[
C_{\varphi}=
\begin{pmatrix}
I_{\mathbb{C}} & \Lambda_{\varphi(0)}\\
0 & \widetilde C_{\varphi}%
\end{pmatrix}
, \qquad\Lambda_{\varphi(0)}g=g(\varphi(0)).
\]
Point evaluation is bounded on $\mathcal{D}_{\alpha}$. Hence this block
representation shows both that boundedness of $C_{\varphi}$ implies
boundedness of $\widetilde C_{\varphi}$ and that the converse holds.
Moreover,
\[
C_{\varphi}-
\begin{pmatrix}
0 & 0\\
0 & \widetilde C_{\varphi}%
\end{pmatrix}
=
\begin{pmatrix}
I_{\mathbb{C}} & \Lambda_{\varphi(0)}\\
0 & 0
\end{pmatrix}
\]
has finite rank. Therefore
\[
\|C_{\varphi}\|_{e}=\|\widetilde C_{\varphi}\|_{e}.
\]
The unitary equivalence gives $\|\widetilde C_{\varphi}\|_{e}=\|D_{\varphi
}\|_{e}$, completing the proof.
\end{proof}

Let $\mu$ be a positive Borel measure on $\mathbb{D}$ such that the embedding
\[
A_{\alpha}^{2}\hookrightarrow L^{2}(\mu)
\]
is bounded. The Toeplitz operator $T_{\mu}$ is the unique bounded positive
operator on $A_{\alpha}^{2}$ satisfying
\[
\langle T_{\mu}f,g\rangle_{A_{\alpha}^{2}} = \int_{\mathbb{D}}f(w)\overline
{g(w)}\,d\mu(w), \qquad f,g\in A_{\alpha}^{2}.
\]
When $d\mu=q\,dA_\alpha$, we write $T_q=T_\mu$.

\begin{lemma}
\label{lem3.3} Let $\alpha>-1$, and let $\varphi:\mathbb{D}\to\mathbb{D}$ be a
nonconstant analytic self-map. Suppose that $C_{\varphi}$ is bounded on
$\mathcal{D}_{\alpha}$. Define
\[
\psi(w) = \frac{N_{\varphi,\alpha}(w)} {(1-|w|^{2})^{\alpha}}, \qquad
w\in\mathbb{D},
\]
and let $\widehat{\mu}_{\varphi,\alpha}$ be the positive Borel measure given
by
\[
d\widehat{\mu}_{\varphi,\alpha}(w) = \psi(w)\,dA_{\alpha}(w).
\]
Then
\[
D_{\varphi}^{*}D_{\varphi}= T_{\widehat{\mu}_{\varphi,\alpha}} \quad\text{on
}A_{\alpha}^{2}.
\]
Consequently,
\[
\|D_{\varphi}\|_{e}^{2} = \|T_{\widehat{\mu}_{\varphi,\alpha}}\|_{e}.
\]

\end{lemma}

\begin{proof}
By Proposition~\ref{prop3.2}, the operator $D_{\varphi}$ is bounded on
$A_{\alpha}^{2}$. Since
\[
dA_{\alpha}(w) = (1+\alpha)(1-|w|^{2})^{\alpha}\,dA(w),
\]
we have
\begin{align*}
d\widehat{\mu}_{\varphi,\alpha}(w)  &  = \psi(w)\,dA_{\alpha}(w)\\
&  = \frac{N_{\varphi,\alpha}(w)} {(1-|w|^{2})^{\alpha}} (1+\alpha
)(1-|w|^{2})^{\alpha}\,dA(w)\\
&  = (1+\alpha)N_{\varphi,\alpha}(w)\,dA(w).
\end{align*}

For every $g\in A_\alpha^2$, the area formula gives
\[
\int_{\mathbb{D}}|g(w)|^2\,d\widehat{\mu}_{\varphi,\alpha}(w)
=\|D_\varphi g\|_{A_\alpha^2}^2
\leq\|D_\varphi\|^2\|g\|_{A_\alpha^2}^2.
\]
Thus $\widehat{\mu}_{\varphi,\alpha}$ induces a bounded embedding, and the associated positive Toeplitz operator is well defined.

For $f,g\in A_{\alpha}^{2}$, by the definition of $T_{\widehat{\mu}%
_{\varphi,\alpha}}$,
\begin{align*}
\left\langle T_{\widehat{\mu}_{\varphi,\alpha}}f,g \right\rangle _{\alpha}  &
= \int_{\mathbb{D}} f(w)\overline{g(w)} \,d\widehat{\mu}_{\varphi,\alpha}(w)\\
&  = (1+\alpha) \int_{\mathbb{D}} f(w)\overline{g(w)} N_{\varphi,\alpha
}(w)\,dA(w).
\end{align*}

On the other hand, by the definition of $D_{\varphi}$ and the change of
variables formula for the generalized Nevanlinna counting function,
\begin{align*}
\left\langle D_{\varphi}^{*}D_{\varphi}f,g \right\rangle _{\alpha}  &  =
\left\langle D_{\varphi}f,D_{\varphi}g \right\rangle _{\alpha}\\
&  = \int_{\mathbb{D}} f(\varphi(z)) \overline{g(\varphi(z))} |\varphi^{\prime}(z)|^{2}\,dA_{\alpha}(z)\\
&  = (1+\alpha) \int_{\mathbb{D}} f(\varphi(z)) \overline{g(\varphi(z))}
|\varphi^{\prime}(z)|^{2}(1-|z|^{2})^{\alpha}\,dA(z)\\
&  = (1+\alpha) \int_{\mathbb{D}} f(w)\overline{g(w)} N_{\varphi,\alpha
}(w)\,dA(w).
\end{align*}
Therefore,
\[
\left\langle D_{\varphi}^{*}D_{\varphi}f,g \right\rangle _{\alpha}=
\left\langle T_{\widehat{\mu}_{\varphi,\alpha}}f,g \right\rangle _{\alpha}%
\]
for all $f,g\in A_{\alpha}^{2}$, and hence
\[
D_{\varphi}^{*}D_{\varphi}= T_{\widehat{\mu}_{\varphi,\alpha}}.
\]

It remains to prove the equality of the essential norms. Since $\mathcal{K}%
(A_{\alpha}^{2})$, the set of compact operators on $A_{\alpha}^{2}$, is a
closed two-sided $*$-ideal in $\mathcal{B}(A_{\alpha}^{2})$ (the set of
bounded operators on $A_{\alpha}^{2}$), the quotient $\mathcal{B}(A_{\alpha
}^{2})/\mathcal{K}(A_{\alpha}^{2}) $ is a $C^{*}$-algebra. Let
\[
\pi:\mathcal{B}(A_{\alpha}^{2}) \longrightarrow\mathcal{B}(A_{\alpha}%
^{2})/\mathcal{K}(A_{\alpha}^{2})
\]
be the canonical quotient map. Then
\[
\|D_{\varphi}\|_{e} = \|\pi(D_{\varphi})\|.
\]
By the $C^{*}$-identity and the fact that $\pi$ is a $*$-homomorphism,
\begin{align*}
\|D_{\varphi}\|_{e}^{2}  &  = \|\pi(D_{\varphi})\|^{2} = \|\pi(D_{\varphi
})^{*}\pi(D_{\varphi})\|\\
&  = \|\pi(D_{\varphi}^{*}D_{\varphi})\| = \|D_{\varphi}^{*}D_{\varphi}\|_{e}
= \|T_{\widehat{\mu}_{\varphi,\alpha}}\|_{e}.
\end{align*}

This completes the proof.
\end{proof}

Let $dA$ denote the normalized Lebesgue measure on $\mathbb{D}$, let
$H(\mathbb{D})$ be the space of all holomorphic functions on $\mathbb{D}$, and
let
\[
L_{a}^{2}(\mathbb{D},dA) = L^{2}(\mathbb{D},dA)\cap H(\mathbb{D})
\]
be the Bergman space on $\mathbb{D}$. Its reproducing kernel is given by
\[
K_{z}(w) = \frac{1}{(1-\overline zw)^{2}}, \qquad z,w\in\mathbb{D}.
\]

For $\gamma\in\mathbb{R}$, define the measure
\[
dV_{\gamma}(z) = K_{z}(z)^{-\gamma}\,dA(z) = (1-|z|^{2})^{2\gamma}\,dA(z).
\]
Then the weighted Bergman space
\[
L_{a}^{2}(\mathbb{D},dV_{\gamma}) = L^{2}(\mathbb{D},dV_{\gamma})\cap
H(\mathbb{D})
\]
is nontrivial whenever
\[
\gamma>-\frac12.
\]
The reproducing kernel of $L_{a}^{2}(\mathbb{D},dV_{\gamma})$ is
\[
K_{z}^{\gamma}(w) = \frac{2\gamma+1} {(1-\overline zw)^{2+2\gamma}}, \qquad
z,w\in\mathbb{D}.
\]

The following lemma is the unit disk case of S. Yamaji's essential norm
estimate \cite[Theorem~B]{MR3127818} for positive Toeplitz operators.
\begin{lemma}
\label{lem3.4} Let $\gamma>-\frac12$, and let $\mu$ be a positive Borel
measure on $\mathbb{D}$. Suppose that the Toeplitz operator $T_{\mu}$ is
bounded on $L_{a}^{2}(\mathbb{D},dV_{\gamma})$. Then, for every fixed $r>0$,
\[
\|T_{\mu}\|_{e} \asymp\limsup_{|z|\to1^{-}}\widetilde{\mu}_{r}(z),
\]
where
\[
\widetilde{\mu}_{r}(z) = \frac{\mu(D(z,r))} {V_{\gamma}(D(z,r))}
\]
is the averaging function of $\mu$.
\end{lemma}

\begin{proof}
The unit disk is a minimal bounded homogeneous domain, and S. Yamaji's weighted measure on the disk is a positive constant multiple of $dV_\gamma$ for every $\gamma>-\frac12$. The area normalization changes only the comparison constants, while the distance normalization rescales the fixed radius. Thus \cite[Theorem~B]{MR3127818} gives the asserted essential norm estimate with closed balls for every fixed radius.

To pass to the open balls used here, observe that
\[
\overline{D(z,r/2)}\subset D(z,r)\subset\overline{D(z,r)}.
\]
For fixed $r>0$, the weighted volumes satisfy $V_\gamma(D(z,r/2))\asymp V_\gamma(D(z,r))$ uniformly in $z$, and each ball has the same weighted volume as its closure. Applying the closed-ball estimate at radii $r/2$ and $r$ therefore yields
\[
\|T_\mu\|_e\asymp\limsup_{|z|\to1^-}\frac{\mu(D(z,r))}{V_\gamma(D(z,r))}=\limsup_{|z|\to1^-}\widetilde{\mu}_r(z),
\]
as required.
\end{proof}

Set $\gamma=\alpha/2$. Then
\[
dA_{\alpha}=(1+\alpha)dV_{\gamma}, \qquad V_{\gamma}(E)=\frac{1}{1+\alpha
}A_{\alpha}(E).
\]
Let
\[
d\nu_{\varphi,\alpha}(w) = \psi(w)\,dV_{\gamma}(w) = N_{\varphi,\alpha
}(w)\,dA(w).
\]
The reproducing kernels for $L_{a}^{2}(\mathbb{D},dV_{\gamma})$ and
$A_{\alpha}^{2}$ differ by the factor $1+\alpha$. Consequently, under the
canonical identification of these two Hilbert spaces, the Toeplitz operator
$T_{\nu_{\varphi,\alpha}}$ on $L_{a}^{2}(\mathbb{D},dV_{\gamma})$ is the same
linear operator as $T_{\widehat\mu_{\varphi,\alpha}}$ on $A_{\alpha}^{2}$.
Indeed,
\begin{align*}
T_{\nu_{\varphi,\alpha}}f(z)  &  = \int_{\mathbb{D}} \frac{1+\alpha
}{(1-z\overline w)^{2+\alpha}} f(w)\psi(w)\,dV_{\gamma}(w)\\
&  = \int_{\mathbb{D}} \frac{f(w)}{(1-z\overline w)^{2+\alpha}} \,d\widehat
\mu_{\varphi,\alpha}(w)\\
&  = T_{\widehat\mu_{\varphi,\alpha}}f(z).
\end{align*}

\medskip Moreover,
\[
\frac{\nu_{\varphi,\alpha}(D(z,r))}{V_{\gamma}(D(z,r))} = (1+\alpha)
\frac{\displaystyle\int_{D(z,r)}N_{\varphi,\alpha}(w)\,dA(w)} {A_{\alpha
}(D(z,r))}.
\]
Since the constant factor $1+\alpha$ is harmless in a two-sided estimate,
Lemma~\ref{lem3.4} applies to $T_{\widehat\mu_{\varphi,\alpha}}$ and yields
the following theorem.

\begin{theorem}
\label{thm3.5} Let $\alpha>-1$, and let $\varphi:\mathbb{D}\to\mathbb{D}$ be a
nonconstant analytic self-map. Suppose that $C_{\varphi}$ is bounded on
$\mathcal{D}_{\alpha}$. Then, for every fixed $r>0$,
\[
\|C_{\varphi}\|_{e}^{2} = \|\widetilde C_{\varphi}\|_{e}^{2} = \|D_{\varphi
}\|_{e}^{2} = \|T_{\widehat{\mu}_{\varphi,\alpha}}\|_{e} \asymp\limsup
_{|z|\to1^{-}} \frac{ \displaystyle
\int_{D(z,r)} N_{\varphi,\alpha}(w)\,dA(w) }{
A_{\alpha}(D(z,r)) }.
\]
The constants in this estimate depend only on $\alpha$ and $r$,
and are independent of $z$.
\end{theorem}

\begin{proof}
By Proposition~\ref{prop3.2},
\[
\|C_{\varphi}\|_{e} = \|\widetilde C_{\varphi}\|_{e} = \|D_{\varphi}\|_{e}.
\]
By Lemma~\ref{lem3.3},
\[
\|D_{\varphi}\|_{e}^{2} = \|T_{\widehat{\mu}_{\varphi,\alpha}}\|_{e}.
\]
The remaining two-sided estimate follows from Lemma~\ref{lem3.4} applied with
$\gamma=\alpha/2$.
\end{proof}

The same Toeplitz representation also gives an exact formula in terms
of the norm of the embedding restricted to an outer annulus.

\begin{proposition}
\label{prop3.6} Let $\alpha>-1$, and let $\varphi$ be a nonconstant
analytic self-map of $\mathbb{D}$ such that $C_\varphi$ is bounded on
$\mathcal{D}_\alpha$. Then
\begin{equation}
\label{eq:3.1}
\|C_\varphi\|_e^2
=(1+\alpha)\lim_{R\to1^-}
\sup_{\substack{g\in A_\alpha^2\\\|g\|_{A_\alpha^2}=1}}
\int_{|w|>R}|g(w)|^2N_{\varphi,\alpha}(w)\,dA(w).
\end{equation}
\end{proposition}

\begin{proof}
Write $\mu=\widehat\mu_{\varphi,\alpha}$ and consider the bounded
embedding $J:A_\alpha^2\to L^2(\mu)$, $Jg=g$. Since
$J^*J=T_\mu$, the $C^*$-identity for essential norms and
Theorem~\ref{thm3.5} give
\[
\|J\|_e^2=\|T_\mu\|_e=\|C_\varphi\|_e^2.
\]

For $0<R<1$, define $M_R:L^2(\mu)\to L^2(\mu)$ by
\[
(M_Rh)(w)=\mathbf1_{\{|w|>R\}}h(w),
\]
and set $J_R=M_RJ$. Let $\{g_n\}$ be a bounded sequence in $A_\alpha^2$. By Montel's
theorem, a subsequence $\{g_{n_j}\}$ converges locally uniformly in
$\mathbb{D}$ to a function $g$, and Fatou's lemma gives
$g\in A_\alpha^2$. Since $J$ is bounded,
$\mu(\mathbb{D})=\|J1\|_{L^2(\mu)}^2<\infty$. Hence
\[
\begin{aligned}
\|(J-J_R)(g_{n_j}-g)\|_{L^2(\mu)}^2
&=\int_{|w|\leq R}|g_{n_j}(w)-g(w)|^2\,d\mu(w)\\
&\leq\mu(\mathbb{D})
\sup_{|w|\leq R}|g_{n_j}(w)-g(w)|^2
\longrightarrow0.
\end{aligned}
\]
Thus $J-J_R$ is compact, and consequently
\[
\|J\|_e\leq\|J-(J-J_R)\|=\|J_R\|.
\]

Conversely, for every compact operator $K:A_\alpha^2\to L^2(\mu)$,
\[
\|J_R\|\leq\|M_R(J-K)\|+\|M_RK\|
\leq\|J-K\|+\|M_RK\|.
\]
As $R\to1^-$, the contractions $M_R$ converge strongly to zero.
This convergence is uniform on the relatively compact image of the
unit ball under $K$, so $\|M_RK\|\to0$. Since $\|J_R\|$ decreases
with $R$, taking the limit and then the infimum over $K$ yields
\[
\lim_{R\to1^-}\|J_R\|\leq\|J\|_e.
\]
Thus $\|J\|_e=\lim\limits_{R\to1^-}\|J_R\|$. This proves \eqref{eq:3.1}.
\end{proof}

We next obtain an exact local average formula under a vanishing
oscillation condition on the counting density.
Under this condition, boundedness of $C_\varphi$ ensures that the
density is bounded near the unit circle.

\begin{definition}
\label{def3.7}
Let $q:\mathbb{D}\to[0,+\infty]$ be a Borel function which is finite
on some outer annulus. We write
$q\in\mathrm{VO}_{\partial}(\mathbb{D})$ if
\begin{equation}
\label{eq:3.2}
\lim_{|z|\to1^-}\sup_{w\in D(z,s)}|q(w)-q(z)|=0
\end{equation}
for every fixed $s>0$. In this case, $q$ is said to have vanishing oscillation
at the boundary.
\end{definition}

For a nonnegative Borel function $q$ such that $q\,dA_\alpha$ induces
a bounded embedding of $A_\alpha^2$, write
\[
\mathcal{B}_\alpha q(z)
=\int_{\mathbb{D}}|\kappa_z^\alpha(w)|^2q(w)\,dA_\alpha(w),
\qquad
\kappa_z^\alpha(w)
=\frac{(1-|z|^2)^{(\alpha+2)/2}}{(1-\overline zw)^{\alpha+2}}.
\]
Thus $\mathcal{B}_\alpha q$ is the Berezin transform of the measure
$q\,dA_\alpha$ on $A_\alpha^2$.

\begin{theorem}
\label{thm3.8}
Let $\alpha>-1$, and let $\varphi$ be a nonconstant analytic
self-map of $\mathbb{D}$ such that $C_\varphi$ is bounded on
$\mathcal{D}_\alpha$. Suppose that
$\psi(w)=\frac{N_{\varphi,\alpha}(w)}{(1-|w|^2)^\alpha}
\in\mathrm{VO}_{\partial}(\mathbb{D})$. Then
\begin{equation}
\label{eq:3.3}
\|C_\varphi\|_e^2
=\limsup_{|z|\to1^-}\psi(z)
=\limsup_{|z|\to1^-}\mathcal{B}_\alpha\psi(z).
\end{equation}
Moreover, for every fixed $r>0$,
\begin{equation}
\label{eq:3.4}
\|C_\varphi\|_e^2
=(1+\alpha)\limsup_{|z|\to1^-}
\frac{\displaystyle\int_{D(z,r)}N_{\varphi,\alpha}(w)\,dA(w)}
{A_\alpha(D(z,r))}.
\end{equation}
\end{theorem}

\begin{proof}
We first show that $\psi$ is bounded on some outer annulus.
By Lemma~\ref{lem3.3}, $T_{\widehat\mu_{\varphi,\alpha}}$ is bounded.
Fix $s_0>0$. Since $\psi\in\mathrm{VO}_{\partial}(\mathbb{D})$,
there exists $\rho\in(0,1)$ such that
\[
\psi(w)\geq\psi(z)-1,\qquad w\in D(z,s_0),\quad \rho<|z|<1.
\]
The normalized kernels have unit norm, so \eqref{eq:2.10}
gives
\[
\int_{D(z,s_0)}|\kappa_z^\alpha(w)|^2\,dA_\alpha(w)
=1-(1-\tanh^2s_0)^{\alpha+1}=:c_{s_0}>0.
\]
Consequently, for $\rho<|z|<1$,
\[
\begin{aligned}
c_{s_0}(\psi(z)-1)
&\leq\int_{D(z,s_0)}|\kappa_z^\alpha(w)|^2\psi(w)\,dA_\alpha(w)\\
&\leq\mathcal{B}_\alpha\psi(z)\\
&=\langle T_{\widehat\mu_{\varphi,\alpha}}\kappa_z^\alpha,
\kappa_z^\alpha\rangle_\alpha
\leq\|T_{\widehat\mu_{\varphi,\alpha}}\|.
\end{aligned}
\]
Thus
\[
\sup\limits_{\rho<|z|<1}\psi(z)
\leq1+\frac{\|T_{\widehat\mu_{\varphi,\alpha}}\|}{c_{s_0}}<\infty.
\]
We may therefore define
\[
q(w)=
\begin{cases}
\psi(w),&\rho<|w|<1,\\
0,&|w|\leq\rho.
\end{cases}
\]
By the compactness of the interior restriction used in
Proposition~\ref{prop3.6}, the operator
\[
T_{\widehat\mu_{\varphi,\alpha}}-T_q
=J^*\mathbf1_{\{|w|\leq\rho\}}J
\]
is compact, where $J:A_\alpha^2\to L^2(\widehat\mu_{\varphi,\alpha})$
is the inclusion map. Thus $\|C_\varphi\|_e^2=\|T_q\|_e$.

For fixed $s>0$, the disk $D(z,s)$ is contained in
$\{|w|>\rho\}$ when $|z|$ is sufficiently close to $1$.
Thus $q\in\mathrm{VO}_{\partial}(\mathbb{D})$.

Put $B=\sup\limits_{w\in\mathbb{D}}q(w)<\infty$. The normalized kernels
have unit norm, and the calculation in
\eqref{eq:2.10} gives
\[
\int_{\mathbb{D}\setminus D(z,s)}
|\kappa_z^\alpha(w)|^2\,dA_\alpha(w)
=(1-\tanh^2s)^{\alpha+1}.
\]
Splitting the Berezin integral over $D(z,s)$ and its complement,
we obtain
\[
|\mathcal{B}_\alpha q(z)-q(z)|
\leq\sup_{w\in D(z,s)}|q(w)-q(z)|
+2B(1-\tanh^2s)^{\alpha+1}.
\]
Since $q\in\mathrm{VO}_{\partial}(\mathbb{D})$ and $\alpha>-1$,
first letting $|z|\to1^-$ with $s$ fixed and then letting
$s\to\infty$ gives $\mathcal{B}_\alpha q(z)-q(z)\to0$.

The normalized kernels $\kappa_z^\alpha$ converge weakly to zero in
$A_\alpha^2$ as $|z|\to1^-$. Therefore, for every compact operator
$K$ on $A_\alpha^2$,
\[
\limsup_{|z|\to1^-}\mathcal{B}_\alpha q(z)
=\limsup_{|z|\to1^-}
\bigl|\langle(T_q-K)\kappa_z^\alpha,
\kappa_z^\alpha\rangle_\alpha\bigr|
\leq\|T_q-K\|.
\]
Taking the infimum over $K$ gives
$\limsup\limits_{|z|\to1^-}q(z)\leq\|T_q\|_e$.

Conversely, $T_{q\mathbf1_{\{|z|\leq R\}}}$ is compact for
$0<R<1$. Hence
\[
\begin{aligned}
\|T_q\|_e
&\leq\|T_q-T_{q\mathbf1_{\{|z|\leq R\}}}\|\\
&=\sup_{\|f\|_{A_\alpha^2}=1}
\int_{|z|>R}q(z)|f(z)|^2\,dA_\alpha(z)\\
&\leq\sup_{|z|>R}q(z).
\end{aligned}
\]
Letting $R\to1^-$ proves
$\|T_q\|_e=\limsup\limits_{|z|\to1^-}q(z)$.
Since $q=\psi$ near the boundary and
$T_{\widehat\mu_{\varphi,\alpha}}-T_q$ is compact,
\[
\mathcal{B}_\alpha\psi(z)-\mathcal{B}_\alpha q(z)
=\langle(T_{\widehat\mu_{\varphi,\alpha}}-T_q)
\kappa_z^\alpha,\kappa_z^\alpha\rangle_\alpha\longrightarrow0.
\]
This proves \eqref{eq:3.3}.

Finally,
\[
(1+\alpha)
\frac{\displaystyle\int_{D(z,r)}N_{\varphi,\alpha}(w)\,dA(w)}
{A_\alpha(D(z,r))}
=\frac{\displaystyle\int_{D(z,r)}\psi(w)\,dA_\alpha(w)}
{A_\alpha(D(z,r))}.
\]
Since $\psi\in\mathrm{VO}_{\partial}(\mathbb{D})$,
\[
\limsup_{|z|\to1^-}
\frac{\displaystyle\int_{D(z,r)}\psi(w)\,dA_\alpha(w)}
{A_\alpha(D(z,r))}
=\limsup_{|z|\to1^-}\psi(z).
\]
Thus \eqref{eq:3.3} gives \eqref{eq:3.4}.
\end{proof}

\medskip\noindent\textbf{Data availability.} This manuscript has no associated data.

\medskip\noindent\textbf{Competing interests.} The authors declare no
competing interests.

\end{document}